\documentclass[a4paper,reqno]{amsart}

\usepackage[all]{xy}
\usepackage{amssymb}
\usepackage{verbatim}
\usepackage{fullpage}
\usepackage{enumerate}
\usepackage[latin1]{inputenc}

\input xy

\xyoption{all}

\xyoption{poly}

\usepackage[active]{srcltx}
\theoremstyle{plain}
\newtheorem{lem}{Lemma}[section]
\newtheorem{thm}{Theorem}[section]
\newtheorem{prop}{Proposition}[section]
\newtheorem{cor}{Corollary}[section]

\theoremstyle{definition}

\newtheorem{rem}{Remark}

\theoremstyle{remark}

 \newcommand{\calC}{\mathcal C} 
 \newcommand{\calF}{\mathcal F}

\newcommand{\calN}{\mathcal N} \newcommand{\calO}{\mathcal O}

 \newcommand{\mC}{\mathbb C} 
 \newcommand{\mF}{\mathbb F}

  \newcommand{\mP}{\mathbb P}
\newcommand{\mQ}{\mathbb Q} \newcommand{\mR}{\mathbb R}

 \newcommand{\mbF}{\mathbf F}

\newcommand{\gra}{\alpha} \newcommand{\grb}{\beta}       
  
      \newcommand{\grs}{\sigma}
\newcommand{\grf}{\varphi}\newcommand{\gro}{\omega}      
\newcommand{\grt} {\theta}

\newcommand{\vuoto}      {\varnothing}

\newcommand{\cech}       {\vee}
\newcommand{\comp}       {\!\circ\!}

\newcommand{\senza}      {\smallsetminus}

\newcommand{\ol}         {\overline}
\newcommand{\wt}         {\widetilde}

\begin{document}

\title{EFFECTIVE AND BIG DIVISORS ON A PROJECTIVE SYMMETRIC VARIETY}

\author[Alessandro Ruzzi]{Alessandro Ruzzi}

%\date{\today}

\curraddr{Dipartimento di Matematica ``Guido
  Castelnuovo"\\ ``Sapienza" Universit\`a di Roma\\ Piazzale Aldo Moro
  5\\ 00185 Roma, Italy}

\email{ruzzi@mat.uniroma1.it; alessandroruzzi78@gmail.com}

\maketitle

%\usepackage{amsmath}
%    \usepackage{amsxtra}
%    \usepackage{amstext}
%    \usepackage{amssymb}
%    \usepackage{latexsym}

%\newcommand{\noapp}{\subset \!\!\!\!\!\!/}
%\newcommand{\real}{\mbox{$I\!\!R$}}
%\newcommand{\notin}{\mbox{$\in \!\!\!\!\!/$}}
%
%\newcounter{rem}
%\setcounter{rem}{1}
%\newcommand{\rem}{\textbf{Remark \arabic{rem}.\ }\addtocounter{rem}{1}}
%
%\newtheorem{prop}{Proposition}[section]
%\newtheorem{thm}{Theorem}[section]
%\newtheorem{thm'}{Theorem}
%\newtheorem{cor}{Corollary}[section]
%\newtheorem{lem}{Lemma}[section]
%
%%\newenvironment{re}{\begin{itemize}\rem}{\end{itemize}}
%
%\newtheorem{cl}{Claim}[section]
%\newtheorem{scl}{Subclaim}[section]
%\newtheorem{es}{Exemple}[section]
%\newtheorem{deff}{Definition}[section]
%\newtheorem{nott}{Notation}
%
%\begin{document}
%\addtolength{\baselineskip}{+.15\baselineskip} \fontsize{10pt}{13pt}

\begin{abstract}
We describe the effective and the big cones of a projective
symmetric variety. Moreover, we give a  necessary and sufficient
combinatorial criterion for the bigness of a nef divisor  on a
projective symmetric variety. When  the
divisor is $G$-stable, such a criterion has an explicit geometric
interpretation. Finally, we describe the spherical closure of a
symmetric subgroup.\\ \textit{keywords}: Symmetric varieties, Big divisors.
\end{abstract}

\

\section*{Introduction}\label{sez: intro}

Brion gave a description of the Picard group of a spherical variety
in \cite{Br89}. He also found necessary and sufficient conditions
for the ampleness and global generation of a line bundle. From these
conditions follows that a line bundle is nef if and only if it is
globally generated. It is natural to ask what are the conditions on
a line bundle to be effective, respectively big. It is known that the
effective cone is closed, polyhedral and, if the variety is
$\mathbb{Q}$-factorial, generated by the classes of the $B$-stable
prime divisors. But in general it is hard to say which are the
$B$-stable prime divisors whose classes generate an extremal ray of
the effective cone. In the very special case of  projective
homogeneous varieties, the big cone coincides with the ample cone.
More generally, the case of wonderful varieties  is studied in
\cite{Br07}. For any normal projective variety, there is a very useful criterion for a nef divisor $D$ to be big: $D$ is big if
and only if its volume $D^{dimX}$ is strictly positive (see \cite{La04}, Theorem 2.2.16). In the case of a toric $T$-variety $Z$ this criterion has a more combinatorial version.
A $T$-stable,  nef  divisor $D$ can be identified with a certain polytope with integral vertices, its moment polytope; moreover,  the volume of $D$ is equal to $(dim\, Z)!$ times the volume of this polytope. In particular, $D$ is big if and only if its moment polytope has positive volume.  See also
\cite{FS} for a partial study of the big cone of some toric
varieties.

We are interested to study the bigness of Cartier
$\mathbb{Q}$-divisors on symmetric varieties (over which acts a
semisimple group). First, we describe explicitly the effective cone;
we determine also when the classes of two $B$-stable prime divisors
are proportional. When the variety is $\mathbb{Q}$-factorial,  we
find the conditions for the class of a $B$-stable prime divisor to
generate an extremal ray of the effective cone   (see Theorem
\ref{eff cone1}, Theorem \ref{eff cone2}  and Corollary \ref{eff cone
of singular X}).

Given a $\mQ$-factorial, projective symmetric variety $X$, we will give an explicit description of its big cone using the description of $Eff(X)$ and the fact the $Big(X)$ is the interior part of $Eff(X)$ (see Theorem \ref{bigg cone}). We will give also a combinatorial version of the cited criterion on the volume of a nef  $\mathbb{Q}$-divisor $D$, stated in terms of the $T$-weights of the fibres of $\calO(D)$ over the $B$-stable points (see Theorem \ref{bigness e nef}). The idea of the proof is the following.
Up to take a lifting and up to linear equivalence, one can reduce itself to study $G$-stable divisors on
a projective toroidal symmetric variety, i.e. a projective variety
such that the closure of none $B$-stable divisor of the open
$G$-orbit contains a $G$-orbit. Each projective toroidal symmetric
variety $X$ (with fixed open $G$-orbit) contains a projective toric
variety $Z^{c}$ which determines $X$ uniquely.  We prove that the
restriction  of a big  $\mathbb{Q}$-divisor $D$ of $X$ to
$Z^{c}$ is always big. If the divisor is $G$-stable, then this
condition is also sufficient (see Proposition \ref{bigness of toroidal}). One can show that the subspace of
$Pic(X)_{\mathbb{Q}}$ generated by the classes of $G$-stable
divisors is a complement to the kernel of the restriction
$Pic(X)_{\mathbb{Q}}\rightarrow Pic(Z^{c})_{\mathbb{Q}}$. Then, we
use the combinatorial description of
$H^{0}(Z^{c},\mathcal{O}(D)|Z^{c})$ to prove Theorem \ref{bigness e nef}.

We  describe also the spherical closure of a symmetric
subgroup $H$ in Proposition \ref{spherical closure}. We will use such description to give a condition so that the class of a $G$-stable prime divisor is proportional to the class of a $B$-stable, but not $G$-stable, divisor (see Corollary \ref{compat spherical closure} and Theorem \ref{eff cone2}). Given a
symmetric space $G/H$, $H$ is called a symmetric subgroup and its
spherical closure is defined, after Luna, as  the subgroup $\ol{H}^{sph}$ of
$N_{G}(H)$ which stabilizes all the $B$-stable prime divisors of
$G/H$. A subgroup $H$ of $G$ is wonderful (resp. spherical) if $G/H$ has a wonderful compactification (resp. is spherical). The definition of spherical closure is useful to associate in a natural way  a wonderful subgroup of $G$ to any spherical subgroup of $G$.   Moreover, $N_G(H)/H$ is isomorphic to the group of $G$-equivariant automorphisms of $G/H$, so $\ol{H}^{sph}$ can be thought as a group of automorphisms.

\section{Notation}\label{sez: notation}

In this section we  introduce the necessary  notations. The reader
interested to the embedding theory of spherical varieties can see
\cite{LV}, \cite{Kn91}, \cite{Br97} or \cite{T06}. In \cite{V90},
this theory is explained in the particular case of  symmetric
varieties.

\subsection{First definitions}\label{sez: 1def}

Let $G$ be a connected semisimple algebraic group over an algebraic closed field  \textbf{k} of characteristic zero
and let $\theta$ be an involution of $G$. Let $H$ be a closed
subgroup of $G$ such that $G^{\theta}\subset H\subset
N_{G}(G^{\theta})$, then we say that   $G/H$ is a symmetric space
and that $H$ is a symmetric subgroup. We can assume $G$ simply
connected (see \cite{V90}, \S2.1). An equivariant embedding of $G/H$ is the data of a
$G$-variety $X$ together with an equivariant open embedding
$\varphi:G/H\hookrightarrow X$, in particular $\varphi(g\cdot
x)=g\cdot\varphi(x)$ for each $x\in G/H$. A normal $G$-variety is
called a spherical variety if it contains a dense orbit under the
action of an arbitrarily chosen Borel subgroup of $G$. One can show
that every normal equivariant embedding of $G/H$ is  spherical (see
\cite{dCP1}, Proposition 1.3); we call it a symmetric variety.   The most important example of symmetric space is the one of a semisimple (or more generally reductive) group $G$ seen as $(G\times G)$-variety.

We
say that a subtorus of $G$ is split if $\theta(t)=t^{-1}$ for all
its elements $t$; moreover it is a maximal split torus if  has
maximal dimension. We say that any maximal torus containing a
maximal split torus is maximally split. Any maximally split torus is
$\theta$ stable; moreover they are all conjugate under $G^{\theta}$ (see \cite{Vu74}, Proposition 2 (iv) and Proposition 6). We fix
arbitrarily a maximally split torus  $T$ containing a maximal split
torus $T^{1}$. Let $R_{G}$ be the root system of $G$ with respect to
$T$. We can choose a Borel subgroup $B$ containing $T$ such that,
given any positive root $\alpha$ with respect to $B$, either
$\theta(\alpha)=\alpha$ or $\theta(\alpha)$ is negative. Moreover,
$BH$ is dense in $G$ (see \cite{dCP1}, Lemma 1.2 and Proposition
1.3).

\subsection{Colored fans}\label{sez: colored fans}

We want to describe the Picard group of a symmetric variety. Before
doing this, we need to introduce some details about the
classification of the symmetric varieties by their colored fans
(this classification holds more generally for spherical varieties).

Let $D(G/H)$ be the set of colors of $G/H$, namely the set of
$B$-stable prime divisors of $G/H$. They are the irreducible
components of $(G/H)\senza (BH/H)$ because $BH/H$ is affine
and open. We say that a spherical variety is simple if it contains
one closed orbit. Let $X$ be a simple symmetric variety with closed
orbit $Y$. We define the set of colors of $X$ as the subset $\calF(X)$
of $D(G/H)$ consisting of the colors whose closure in $X$ contains
$Y$.  To each prime divisor $D$ of $X$, we can associate the
normalized discrete valuation $v_{D}$ of $\mathbb{C}(G/H)$ whose
ring is the local ring $\mathcal{O}_{X,D}$. One can prove that $D$
is $G$-stable if and only if $v_{D}$ is $G$-invariant, i.e.
$v_{D}(g\cdot f)=v_{D}(f)$ for each $g\in G$ and $f\in
\mathbb{C}(G/H)$. Let $\calN$ be the  set of all $G$-invariant
valuations of $\mathbb{C}(G/H)$ taking values in $\mathbb{Z}$ and
let $\calN(X)$ be the set of the valuations associated to the $G$-stable
prime divisors of $X$. Observe that each irreducible component  of
$X\senza (G/H)$ has codimension one, because $G/H$ is
affine. Let $S:=T/\,T\cap H\simeq T\cdot (H/H)$. One can show that
the group $\mathbb{C}(G/H)^{(B)}/\,\mathbb{C}^{*}$ is isomorphic to
the character group $\chi(S)$ of $S$ (see \cite{V90}, \S2.3); in
particular, it is a free abelian group. We define the \textit{rank} of $G/H$
as the dimension  of $S$. We can identify the dual group
$Hom_{\mathbb{Z}}(\mathbb{C}(G/H)^{(B)}/\,\mathbb{C}^{*},\mathbb{Z})$
with the group  $\chi_{*}(S)$ of one-parameter subgroups of $S$.
The  restriction map of valuations to $\mathbb{C}(G/H)^{(B)}/\mathbb{C}^{*}$ is
injective over $\calN$ (see \cite{Kn91}, Corollary 1.8), so we can
identify $\calN$ with a subset of $\chi_{*}(S)$.
We say
that $\calN$ is the valuation monoid of $G/H$. Observe that $\calN$ is
stable under addition (see, for example \cite{Kn91}, Lemma 5.1).
For each color $F$, we define $\rho(F)$ as the restriction of
$v_{F}$ to $\chi(S)$. In general, the map
$\rho:D(G/H)\rightarrow\chi_{*}(S)_{ \mathbb{R}}$ is not injective.
Let $\calC(X)$ be the cone in $\chi_{*}(S)_{ \mathbb{R}}$ generated by
$\calN(X)$ and $\rho(\calF(X))$.
A simple symmetric variety is uniquely determined by its colored
cone $(\calC(X),\calF(X))$ (see \cite{Kn91}, Theorem 3.1).

Let $Y$ be a $G$-orbit of a symmetric variety  $X$. The set $\{x\in
X\ |\  \overline{G\cdot x}\supset Y\}$ is an open simple
$G$-subvariety of $X$ with closed orbit $Y$, because any spherical
variety contains only finitely many  $G$-orbits. Let $\{X_{i}\}$ be
the set of open simple subvarieties of $X$ and define the set of
colors of $X$, $\calF(X)$, as $\bigcup_{i\in I}\calF(X_{i})$. The family
$\mbF(X):=\{(\calC(X_{i}),\calF(X_{i}))\}_{i\in I}$ is \vspace{0.2 mm} called the
colored fan of $X$ and determines completely $X$ (see \cite{Kn91},
Theorem 3.3). Moreover $X$ is complete if and only if $\calN$ is
contained in the support $\bigcup_{i\in I}\calC(X_{i}) $ of $\mbF(X)$  (see \cite{Kn91}, Theorem 4.2).

If one allows $G$ to be reductive, then the toric varieties are a
special case of symmetric varieties.  If $X$ is a toric variety,
then $D(G/H)$ is empty and we need only to consider the fan
$\mbF^f(X)=\{(\calC(X_{i}))\}_{i\in I}$ associated to the colored fan of $X$
(actually the classification of spherical varieties by colored fans
is a generalization of the classification of toric varieties by
fans). One can show that, fixed any symmetric space $G/H$ such that $\rho$ is injective (for example if $G/H$ is a group), the symmetric varieties with open orbit $G/H$ are classified by the fans $\mbF^f(X)=\{(\calC(X_{i}))\}_{i\in I}$.

\subsection{Restricted root system}\label{sez: restricted root system}

To describe  the sets $\calN$ and $\rho(D(G/H))$, we need to associate a
root system  to $G/H$. W can identify $\chi(T^{1} )_{\mathbb{R}}$
with $\chi(S)_{\mathbb{R}}$ because  $\chi(S)$ has finite index in
$\chi(T^{1})$. We call again $\theta$ the involution induced on
$\chi(T)_{\mathbb{R}}$. The inclusion $T^{1}\subset T$ induces an
isomorphism of $\chi(T^{1})_{\mathbb{R}}$ with the $(-1)$-eigenspace
of $\chi(T)_{\mathbb{R}}$ under the action of $\theta$ (see
\cite{T06}, \S 26). Denote  by $W_{G}$ the Weyl group of $G$ (with
respect to $T$). We can identify $\chi(T^{1})_{\mathbb{R}}$ with its
dual $\chi_{*}(T^{1})_{\mathbb{R}}$ by the restriction $(\, \cdot
,\cdot)$  of the Killing form to $\chi(T^{1})_{\mathbb{R}}$ . The
set $R_{G,\theta}:= \{\beta-\theta(\beta)\ |\ \beta\in
R_{G}\}\backslash\,\{0\}$ is a root system in $\chi(S)_{\mathbb{R}}$
(see \cite{V90}, \S 2.3 Lemme), which we call the restricted root
system of $(G,\theta)$; we call the non zero $\beta-\theta(\beta)$
the restricted roots. We denote by $\alpha_{1},...,\alpha_{s}$ the
elements of the basis
$\overline{R}_{G,\theta}:=\{\beta-\theta(\beta)\, |\, \beta\in
R_{G}$ simple$\}\,\backslash\, \{0\}$   of $R_{G,\theta}$. We denote
by $\{\alpha^{\vee}_{1} ,...,\alpha^{\vee}_{s}\}$ the dual basis of
the restricted coroot system $R^{\vee}_{G,\theta}$, i.e. the dual
root
system of $R_{G,\theta}$.
Let $C^{-}$ be the negative Weyl chamber of $R_{G,\grt}$ (in $\chi_*(S)_\mQ$).
Given a dominant weight $\lambda$ of $G$, we denote by $V(\lambda)$
the irreducible representation of highest weight $\lambda$. See
\cite{CM}, Theorem 2.3 or \cite{T06}, Proposition 26.4 for an
explicit description of the dominant weights of $R_{G,\theta}$. They
are  called spherical weights and are also dominant weights of
$R_{G}$. Let $W_{G,\theta }$ be the Weyl group of $R_{G,\theta}$; it
is isomorphic to $N_{H^{0}}(T_{1})/Z_{H^{0}}(T^{1})$ and to
$N_{G}(T_{1})/Z_{G}(T^{1})$ (see Proposition 26.2 in \cite{T06}).

\subsection{The sets $\calN$ and $D(G/H)$}\label{sez: N+Dg/h}

The set $\calN$ is equal to $C^{-}\cap\,\chi_{*}(S)$; in particular, it
consists of the lattice vectors of the rational, polyhedral, convex
cone $C^{-}=cone(\calN)$. The set $\rho(D(G/H))$ is equal to
$\{\alpha^{\vee}_{1} ,...,\alpha^{\vee}_{s}\}$ and, for each $i$,
the fibre $\rho^{-1}(\alpha_{i}^{\vee})$ contains  at most 2 colors.
In particular,  the number of colors of a symmetric space is at
least its rank. We say that $(G,\theta)$ is indecomposable if
any  normal, connected, $\theta$-stable subgroup of $G$ is trivial.
In this case the number of colors of $G/H$ is at most equal to rank
of $G/H$ plus one. If $|D(G/H)|>rank(G/H)$ and $(G,\theta)$ is
indecomposable, we have two possibilities: 1)
$G^{\theta}=H=N_{G}(G^{\theta})$; 2) $H$ is equal to $G^{\theta}$
and has index two in $N_{G}(G^{\theta})$. In the last case any
element of $N_{G}(G^{\theta})\backslash\, G^{\theta}$ exchanges  two
colors. Because of the simply-connectedness of $G$,  we can always write $G$
as a direct product $\prod G_{i}$ of $\theta$-stable, semisimple,
normal subgroup $G_{i}$ such that each $(G_{i},\theta)$ is
indecomposable. Moreover, the finite cover $G/G^{\theta}$ of $G/H$
is the direct product $\prod_{i\in I}G_{i}/G_{i}^{\theta}$.

We say that a simple restricted root $\alpha$ is exceptional if $\rho^{-1}(\alpha^{\vee})$ contains two colors and
$2\gra$ is a restricted root. Moreover, we say that a symmetric variety is exceptional if there is an exceptional root. If moreover $\grt$ is indecomposable, then $G^\grt=H=N_G(G^\grt)$. Furthermore, $\rho$ is injective if $H$ is semisimple.

Let $D(G/H)^{H}$ be the set of colors $F$ such that $\rho^{-1}(\rho(F))=\{F\}$. One can show that  $\rho^{-1}(\rho(F))=\{F\}$ if only if the equation of
$\pi^{-1}(F)$ in $G$  is $H$-invariant, where $\pi:G\rightarrow G/H$
is the projection. We denote by $F_{\alpha}$ the sum of the colors
in $\rho^{-1}(\alpha^{\vee})$. If $\alpha^{\vee}\notin\,
\rho(D(G/H)^{H})$, we write $\rho^{-1}(\alpha
^{\vee})=\{F^{+}_{\alpha},F_{\alpha}^{-}\}$, so
$F_{\alpha}=F^{+}_{\alpha}+F_{\alpha}^{-}$.

\subsection{Toroidal symmetric varieties}\label{sez: toroidal}

Before to describe the Picard group, we want to define a special
class of varieties. We say that a spherical variety is toroidal
if $\calF(X)=\vuoto$. There is a special toroidal completion of any
symmetric space $G/H$, because $N_{G}(H)/H$ is finite. This
completion, called the standard completion $X_{0}$, is the simple
symmetric variety associated to $(cone(\calN),\vuoto)$ and it is the
unique  simple completion of $G/H$ which dominates all the other
simple completions. When $X_{0}$ is smooth then it is a wonderful
variety (in the definition of Luna). In particular, $X_{0}$ is smooth
(and wonderful) if $H=N_{G}(G^{\theta})$. This case has been defined
and studied by De Concini and Procesi in \cite{dCP1}.

$X_{0}$ contains an affine toric $S$-variety $Z_{0}$, which is a
quotient of an affine space by a finite group. The toroidal
varieties are the symmetric varieties which dominate the standard
completion; they are in one-to-one correspondence with the $S$-toric
varieties which dominate  $Z_{0}$. The correspondence is
obtained in the following way. The open set $U:=X_{0}\senza
\bigcup_{D(G/H)}\overline{F}$ is a $B$-stable affine set; let $P$ be
its stabilizer. $U$ is $P$-isomorphic to $R_{u}P\times Z_{0}$, where
$R_{u}P=\prod_{\beta \succ0,\, \theta(\beta)\neq\beta}U_{\beta}$ is
the unipotent radical of $P$ and $dim\, Z_{0}=rank\, X_{0}$. To any
toroidal variety $X$ we associate the inverse image $Z$ of $Z_{0}$
by the projection $X\rightarrow X_{0}$. Moreover,  $X\senza
\bigcup_{D(G/H)}\overline{F}$ is $P$-stable and is $P$-isomorphic to
$R_{u}P\times Z$. The toroidal varieties are also in one-to-one
correspondence with a class of complete toric varieties in the
following way. To a symmetric variety variety $X$, we associate the
closure $Z^{c}$ of $Z$ in $X$; $Z^{c}$ is also the inverse image of
$Z_{0}^{c}$. The $T$-variety $Z^{c}$ can be covered by finitely many
$N_{G^{\theta}}(T^{1})$-translated of $Z$; thus $Z^{c}$ is a
$S$-toric variety, in particular it is normal. The fan of $Z$ is the
fan $\mbF^f(X)$ associated to the colored fan  $\mbF(X)$, while the fan of $Z^{c}$
consists of the translates of the cones of $Z$ by the Weyl group
$W_{G,\theta}$ of $R_{G,\theta}$.

Given a symmetric variety $X$ there is  a unique  minimal  toroidal
variety $X^{dec}$, called the decoloration of $X$, which dominates $X$. If $\mbF(X)=\{(\calC_{i},\calF_{i})\}_{i\in I}$,   the colored fan
of $X^{dec}$ is  $\{(\calC_{i}\cap cone(\calN),\vuoto)\}_{i\in I}$ .

\subsection{Big divisor}

Before to describe the Picard group of a symmetric variety, we define some general notions about the line bundle. The reader can see \cite{La04} for more details. Let $X$ be a (normal) projective algebraic variety
over an algebraically closed field of characteristic zero.  Let $CDiv(X)$ the group of Cartier divisors.  Given two Cartier divisors $D_1$ and $D_2$ on $X$, we say that they  are \textit{numerically equivalent} if $D_1\cdot C=D_1\cdot C$ for each curve $C$ on $X$ (here $\cdot$ is the intersection product). In such a case we say also that $\calO(D_1)$ and $\calO(D_2)$ are numerically equivalent.

We define two generalization of an ample bundle. By Nakai's criterion (see \cite{La04}, Theorem 1.2.23), a Cartier divisor $D$ is ample if and only if the intersection product $D^{dim\,Y}\cdot Y$ is strictly positive for each subvariety $Y$ (with $dim\, Y>0)$. A first generalization is obtained weakening such property. Indeed, we say that a Cartier divisor $D$ is \textit{nef} (or numerically effective) if   $D\cdot C\geq 0$ for each curve $C$ in $X$. Remark that in definition we have used only subvarieties of dimension one. But, by Kleiman's Theorem (see \cite{La04}, Theorem 1.4.9),  $D$ is nef if and only $D^{dim\,Y}\cdot Y\geq0$  for each subvariety $Y$ (with $dim\, Y>0)$.

To define  the second generalization of ample divisor  we need to define the Itaka dimension %2.1.3
of a divisor. Given a Cartier divisor $D$, let $E(D):=\{m\geq0: H^0(X,\calO(mD))\neq0\}$; given any $m\in E(D)$ we have a rational map $\phi_m:X\dashrightarrow \mP(H^0(X,\calO(mD)))$. If $E(D)$ is empty we define the \textit{Itaka dimension}  $\kappa(D)$ of $D$ as $-\infty$. Otherwise we define $\kappa(D):=max_{m\in E(D)}\{dim\,\phi_m(X)\}$. Remarks that $\kappa(D)$ is equal at most to the dimension of $X$.
When $D$ is the canonical divisor, $\kappa(D)$ is also called the Kodaira dimension of $X$.
We  say that a Cartier divisor on $X$ is \textit{big} if and only if its Itaka dimension is equal to the dimension of $X$.
Clearly an ample divisor is big.

We have some equivalent conditions for the bigness of a divisor. First, we recall a lemma.

\begin{lem}[see \cite{La04}, Corollary 2.1.38] Let $D$ be a Cartier divisor on $X$ and let $\kappa=\kappa(D)$. Then there are strictly positive constants $a$ and $A$ such that
\[a\cdot m^\kappa\leq dim\, H^0(X,\calO(mD))\leq A\cdot m^\kappa\]
for all sufficiently large $m\in E(D)$.
\end{lem}

\begin{prop}[see \cite{La04}, Lemma 2.2.3 and Corollary 2.2.7]\label{cns D big} Let $D$ be a Cartier divisor on $X$. The following conditions are equivalent:
\begin{itemize}
\item $D$ is big;
\item there is a constant $C>0$ such that $ dim\, H^0(X,\calO(mD))\geq C\cdot m^{\dim\, X}$ for all sufficiently large $m\in E(D)$;
\item for each ample divisor $A$, there is $m>0$ such that $mD-A$ is linearly equivalent to an effective divisor.
\item for each ample divisor $A$, there is $m>0$ such that $mD-A$ is numerically equivalent to an effective divisor.
\end{itemize}
\end{prop}

When $D$ is nef is a lot easier to verify if it is big:

\begin{prop}[see \cite{La04}, Theorem 2.2.16 and Corollary 1.4.41]\label{cns D big e nef} Let $D$ be a nef divisor on $X$ and let $n$ be the dimension of $X$. Then $D$ is big if and only if $vol(D)=D^{n}$ is strictly positive.
\end{prop}

One can define a \textit{$\mQ$-divisor} (respectively $\mR$-divisor) as an element of $CDiv(X)_\mQ$ (resp. of $CDiv(X)_\mQ$). All the previous definitions can easily extend to $\mQ$-divisors. For example, we say that a $\mQ$-divisor $D$ is big if there is $m>0$ such that $mD$ is a big divisor. One can also extends such definitions to $\mR$-divisors, but we do not do it here because it is a little more technical.
Often one  works in the quotient  $Pic(X)_\mR/\equiv$ of the real Picard group by the numerical equivalence for the following two reasons: such space is finite-dimensional (see \cite{La04}, Proposition 1.1.14) and the previous definitions depends only by the numerical class of a divisor (see \cite{La04},  Corollary 1.2.20  and Corollary 2.2.7). Often, we will work with  $\mQ$-divisors by simplicity.

We define the following cones in $Pic(X)_\mQ/\equiv$. \begin{itemize}
\item $Amp(X)$ is the cone generated by the class of ample divisors;
\item $Nef(X)$ is the cone generated by the class of nef divisors;
\item $Big(X)$ is the cone generated by the class of big divisors;
\item $Eff(X)$ is the cone generated by the class of effective divisors.
\item the pseudo-effective cone $PE(X)$ is the closure of $Big(X)$.
\end{itemize}

Such cone are related in the following way:
\begin{prop} All the previous cone are convex. Moreover:
\begin{itemize}
\item $Amp(X)$ and $Big(X)$ are open;
\item $Nef(X)$ and $PE(X)$ are closed;
\item$[$see \cite{La04}, Theorem 1.4.23 $]$  $Nef(X)$ is the closure of $Amp(X)$ and $Amp(X)$ is the interior part of $Nef(X)$;
\item$[$see \cite{La04}, Theorem 2.2.26 $]$ $PE(X)$ is the closure of $Big(X)$ and $Big(X)$ is the interior part of $PE(X)$;
\item $Big(X)\subset Eff(X)\subset PE(X)$.
\end{itemize}
\end{prop}

We define $Eff(X)$ in the same way when $X$ is only complete (and normal).
Finally we say that $X$ is $\mQ$-factorial if its rational Picard group coincides with the rational class group; clearly the smooth varieties are $\mQ$-factorial.

When $X$ is a spherical variety, two  Cartier divisors are numerically equivalent if and only if they are linearly equivalent (see for example \cite{Br93}, Corollaire 1.3), so we can omit the quotient by $\equiv$. Moreover, any effective divisor is linearly equivalent to a
$B$-stable effective divisor (see for example \cite{Br93},
Th\'{e}or\`{e}me 1.3). Thus, if $X$ is $\mQ$-factorial, the effective cone $Eff(X)$ is the
closed polyhedral cone generated by the class of the $B$-stable
prime divisors (there are a finite number of them); in particular $Eff(X)$ is equal to the
pseudo-effective cone.

\subsection{The Picard group}\label{sez: picard}

The class group of a symmetric variety is generated by the classes
of the $B$-stable prime divisors with the relations $div(f)$, where
$f\in \mathbb{C}(G/H)^{(B)}$. Indeed $Cl(BH/H)=Pic(BH/H)$ is
trivial. Given $\omega\in \chi(S)$ we denote by $f_{\omega}$  the
element of $\mathbb{C}(G/H)^{(B)}$ with weight $\omega$ and such
that $f_{\omega}(H/H)=1$. We denote by $v_{E}$ the image of $E\in
\calN(X)$ in $\chi_{*}(S)$; vice versa, given an element $\omega$ of the
image of $\calN(X)$ in $C^{-}\cap \calN$ we denote by $E_{\omega}$ the
corresponding elements of $\calN(X)$.

A Weil divisor $D=\sum_{F\in D(G/H)} a_{F}F+       \sum_{E\in
\calN(X)}b_{E}E$ is a Cartier divisor if and only if, for each colored
cone $(\calC,\calF)$, there is $h_{\calC}\in \chi(S)$ such that $h_{\calC}(E)=a_{E}$
for each $E\in \calC$ and $h_{\calC}(\rho(F))=a_{F}$  for each $F\in \calF$. The
$h_{\calC}$ define a piecewise linear function on the support of $\mbF(X)$ (see \cite{Br89},  Proposition 3.1). We denote such function by $h^D$ or by $h$; sometimes we will use also the notation $h^D_\calC$ instead of $h_\calC$.

Let $\widetilde{PL}(X)$ be the set of functions $h$ on the support
$\bigcup \calC$ of $\mbF(X)$ such that: 1) $h$ is linear on
each colored cone; 2) $h$ takes integral values at all the point of
 $\chi_{*}(S)\cap (\bigcup \calC)$. Let $PL(X)$ be the quotient of
$\widetilde{PL}(X)$ by the subset  of  restrictions of linear
functions. If $X$ is complete, there is the following  exact sequence (see \cite{Br89}, Theorem 3.1):

\[0\rightarrow \bigoplus_{F\in D(G/H)\senza\,\calF(X)}\mathbb{Z}F\rightarrow Pic(X)\rightarrow PL(X)\rightarrow0.\]

A Cartier divisor is globally generated (resp. ample) if and
only if the associated function is convex (resp. strictly convex)
and $h_{\calC}(\rho(F))\leq a_{F}$ (resp. $h_{\calC}(\rho(F))< a_{F}$) for
each colored cone $(\calC,\calF)$ and each $F\in D(G/H)\senza \calF$ (see \cite{Br89}, Proposition 3.3). In particular a Cartier divisor is nef if and only if it is globally generated. Given any Cartier
divisor $F=\sum_{D(G/H)}n_{F}F+ \sum_{\calN(X)}h(E)E$, then
$H^{0}(X,\mathcal{O}(D))$ is a multiplicity-free representation of
$G$ and its
highest weights are in one-to-one correspondence with the point of
$\chi(S)\cap P(D)$ where $P(D)$ is the polytope in
$\chi(S)_{\mathbb{R}}$ intersection of the following half-spaces (see \cite{Br89}, Theorem 3.3):
i) $\{m: m(E)+h(E)\geq0\}$ for each $E\in \calN(X)$;  ii)   $\{m:
m(F)+n_{F}\geq0\}$ for each $F\in D(G/H)$.  If $D$ is globally
generated then the $h_{\calC}$ belongs to $\chi(S)\cap P(D)$. If moreover $X$ is toroidal and  $D$ is $G$-stable, then $P(D)$ is the
intersection of the positive Weyl chamber $C^{+}$ with the convex
hull
$Q(D)=convex (wh_{\calC})$ of the points $wh_{\calC}$, where $(\calC,\vuoto)$ varies in
the set of maximal colored cone and $w$ varies in $W_{G,\theta}$
(see \cite{Bi90}, Corollary 4.1). Furthermore, the integral points in $Q(D)$ are the weights of a basis of seminvariant vectors of $H^{0}(Z^{c},\mathcal{O}(D)|Z^{c})$ (see \cite{Bi90}, Proposition 4.1) and the volume of $\calO(D)|Z^c$ is equal to $(rank\,G/H)!\ vol(Q(D))$.

\begin{rem}\label{rem: sezioni pullback} Let $\varphi:X\rightarrow X'$ be a $G$-equivariant, birational morphism of symmetric varieties and let $L$ be any line bundle over $X'$, then $H^{0}(X',L)$ is isomorphic to $H^0(X,\grf^*(L))$; in particular, $L$ is big if and only if $\grf^*(L)$ is big. Moreover, $L$ is nef if and only if $\grf^*L$ is nef, because of the previous description of nef divisors.
\end{rem}

When $X$ is toroidal we have the following split exact sequence (see \cite{Br89}, Proposition 3.2):

\[0\rightarrow Pic(X_{0})\rightarrow Pic(X)\rightarrow Pic(Z)\rightarrow0,\]
where the maps are induced respectively by the projection
$X\rightarrow X_{0}$ and by the inclusion $Z\hookrightarrow X$.
Given any simple $X$, its Picard group is isomorphic to $\bigoplus_{F\in D(G/H)\senza \calF(X)}\mathbb{Z}[F]$; in particular $Pic(X_{0})=\bigoplus_{F\in D(G/H)}\mathbb{Z}[F]$.

A (complete) symmetric variety is $\mathbb{Q}$-factorial if and only if each
colored cone is simplicial and $\rho$ is injective over $\calF(X)$ (see \cite{Br93}, Proposition 4.2).
Recall that a cone is said simplicial if it is generated by a number
of vectors equal to its dimension. In particular the standard
completion of any symmetric space is $\mathbb{Q}$-factorial. The
conditions for the smoothness  are much more complicated (see
\cite{Ru2}, Theorem 2.2). Notice that the most part of this section is true for any spherical variety:  in particular the descriptions of the class group and of the Picard group holds in general. Also the previous condition for the $\mQ$-factoriality is stated in \cite{Br93} in a more general form which holds for all spherical varieties.

\section{Spherical closure}\label{sez: spher clos}

We define (after Luna) the spherical closure $\overline{H}^{sph}$ of
$H$ as the subgroup of $N_{G}(H)$ which stabilizes all the colors of
$G/H$. The standard completion $X_{0}^{sph}$ of
$G/\overline{H}^{sph}$ is wonderful and the standard completion $X_{0}$ of
$G/H$ is a ramified cover of  $X_{0}^{sph}$. Moreover the
projection induces an isomorphism between their rational Picard
groups. Indeed, we can identify $D(G/H)$ with
$D(G/\overline{H}^{sph})$.

\begin{prop}\label{spherical closure}
Let $G/H$ a symmetric space, then we can write
$(G,\theta)=\prod_{i=1}^{n}(G_{i},\theta)$ so that $G/H$ is a direct
product $\prod_{i=1}^n G_{i}/(G_{i}\cap H)$, where the $(G_{i},\theta)$ with
$i>1$ are indecomposable and $|D(G_{i}/(G_{i}\cap
H))|>rank(G_{i}/(G_{i}\cap H))$ if and only if $i>1$. Moreover
$\overline{H}^{sph}=N_{G_{1}}(G_{1}\cap H)\times
\prod_{i=2}^{n}(G_{i}\cap H)$.
\end{prop}

To prove such proposition we will use \S\ref{sez: N+Dg/h}. In particular, we will use the following fact: if there is $n\in N(H)$ and $D_1$, $D_2\in D(G/H)$ such that $nD_1=D_2$, then $\rho(D_1)=\rho(D_2)$.

\begin{proof} First, we  reduce to the non-exceptional case. Write
$(G,\theta)=\prod_{i=1}^{n}$ $(G_{i},\theta)$ with $(G_{1},\theta)$ non
exceptional and the other $(G_{i},\theta)$  indecomposable and
exceptional. For each $i>1$, we have $G_{i}^{\theta}=G_{i}\cap H
=N_{G_{i}}(G_{i}^{\theta})$ and $|D(G_{i}/(G_{i}\cap
H))|=rank(G_{i}/(G_{i}\cap H))$+1; in particular the $G_{i}\cap H$
with $i>1$ are spherically closed in $ G_{i}$ and $G/H$ is the
direct product $\prod G_{i}/(G_{i}\cap H)$.

Suppose now $X$ non-exceptional and write $(G,\theta)=\prod
(G_{i},\theta)$ with the $(G_{i},\theta)$ indecomposable. Let
$H_{i}:=G_{i}\cap H$. We can think of $D(G_{i}/G_{i}^{\theta})$ as a
subset of $D(G/G^{\theta})$ by associating $F\times \prod_{j\neq
i}G_{j}/G_{j}^{\theta}\in D(G/G^{\theta})$ to any $F\in
D(G_{i}/G_{i}^{\theta})$.  We can suppose that: 1) $|D(G_{i}/
G_{i}^{\theta})|>rank(G_{i}/ G_{i}^{\theta})$ if and only if $i> r$;
2) there is $h\in H$ which exchanges two colors of
$G_{i}/G_{i}^{\theta}$ if and only if $r<i\leq r+m$. Observe that if
$i>r$ there is always an element of $N_{G}(H)$ which exchanges two
colors of $G_{i}/G_{i}^{\theta}$. Let $G_{0}=\prod_{i=1}^{r+m} G_{i}
$, then $\overline{H}^{sph}$ is contained in $N':=
N_{G_{0}}(H_{0})\times \prod_{i>r+m} H_{i}$ because any element of
$N_{G}(H)\senza N'$ exchanges two colors of some
$G_{i}/G_{i}^{\theta}$ with $i>r+m$ (which correspond to two
distinct colors of $G/H$). Moreover, the number of colors of
$G_{0}/ H_{0}$ is equal to its rank. But $|D(G_{0}/ H_{0})|\geq
|D(G_{0}/ N_{G_{0}}(H_{0}))|\geq rank(G_{0}/
N_{G_{0}}(H_{0}))=rank(G_{0}/ H_{0})$, thus the spherical closure of
$H_{0}$ in $G_{0}$ is $N_{G_{0}}(H_{0})$. Therefore
$\overline{H}^{sph}\supset N_{G_{0}}(H_{0})\times
\prod_{i>r+m}H_{i}$.
\end{proof}

\begin{cor} \label{compat spherical closure} If $H$ is spherically closed, then $G/H$ is a direct product of
indecomposable symmetric spaces $\prod G_{i}/H_{i}$. Moreover, the
wonderful completion of $G/H$ is the product of the wonderful
completions of the $G_{i}/H_{i}$.
\end{cor}

\begin{rem}\label{rem: fan spherical closure} Let $X$ be a symmetric variety  with open orbit $G/H$.
Then there is a unique maximal symmetric variety $X^{sph}$ with open
orbit $G/\overline{H}^{sph}$ and with an equivariant proper morphism
$X\rightarrow X^{sph}$ that extends the canonical projection
$G/H\rightarrow G/\overline{H}^{sph}$. Indeed, we can identify
$D(G/H)$ with $D(G/\overline{H}^{sph})$, respectively
$(\mathbb{C}(G/H)^{(B)}/\mathbb{C}^{*})_{\mathbb{Q}}$ with
$(\mathbb{C}(G/\overline{H}^{sph})^{(B)}/\mathbb{C}^{*})_{\mathbb{Q}}$.
Thus the colored fan of $X$ defines a colored fan associated to an
embedding of $G/\overline{H}^{sph}$. It is easy to show that this
variety satisfies the requested properties (see also \cite{Kn91}, \S4).
\end{rem}

\section{Effective cone of a complete symmetric variety}\label{sez: eff cone}

%Given any complete algebraic  variety $X$ we define $Eff(X)$ as the cone in $Pic(X)_{\mathbb{R}}/\equiv$
%generated by the classes of effective divisors (here $\equiv$ is the numerical equivalence). If $X$ is also projective, we define  $Big(X)$ as the cone in $Pic(X)_{\mathbb{R}}/\equiv$ generated by the classes of big divisors. Finally, we define the pseudo-effective cone $PE(X)$ as the closure of $Eff(X)$; $PE(X)$ is also the closure of $Big(X)$, while $Big(X)$ is the interior part of $PE(X)$. When $X$ is a spherical variety, two  Cartier divisors are numerically equivalent if and only if they are linearly equivalent (see for example \cite{Br93},
%Corollaire 1.3), so we can omit the quotient by $\equiv$. Moreover, any effective divisor is linearly equivalent to a
%$B$-stable effective divisor (see for example \cite{Br93},
%Th\'{e}or\`{e}me 1.3). Thus, if $X$ is $\mQ$-factorial, the effective cone $Eff(X)$ is the
%closed polyhedral cone generated by the class of the $B$-stable
%prime divisors (there are a finite number of them); in particular $Eff(X)$ is equal to the
%pseudo-effective cone.

First, we determine $Eff(X)$ when $X$ is
$\mathbb{Q}$-factorial, then we consider the case where $X$ is
projective but possibly not $\mathbb{Q}$-factorial.

\begin{thm}\label{eff cone1}
Let $X$ be a $\mathbb{Q}$-factorial complete  symmetric variety.
Then:
\begin{enumerate}
\item  $Eff(X)$ is a closed polyhedral cone whose extremal rays
are generated by:
\begin{itemize}\item the colors which are not contained in
$D(G/H)^{H}$;
\item the $G$-stable prime divisors $E$ whose classes are not proportional
to any of the $[F^{+}_{\alpha_{i}}+F^{-}_{\alpha_{i}}]$.
\end{itemize}
\item given any extremal ray $r$ of $Eff(X)$, there exist a unique prime  divisor $D$ which belongs to one of the two previous families and
such that $[D]\in r$.
\end{enumerate}\end{thm}

Let $\widetilde{Div}\, ^{G}(X)$ be the abelian group freely
generated by the $G$-stable prime divisors and let $Div^{G}(X)$ be
its image in $Pic(X)$. These groups are isomorphic; indeed  there is
no non-trivial  relation between these divisors, because any rational  function which is a $G$-eigenvector  is constant. If $X$ is toroidal,
we can identify $\widetilde{Div}\, ^{G}(X)_{\mathbb{R}} $ with
$\widetilde{Div}\, ^{T}(Z)_{\mathbb{R}} $ by the restriction.

To prove the theorem we use the explicit knowledge of the relations of $Cl(X)$. In particular, the principal divisor associated to any function in $\mathbb{C}(X)^{(B)}$ is a linear
combination of the $F_{\alpha}$ and of the $G$-stable prime
divisors. Moreover, any [$F_{\alpha}]$ belongs to
$Div^{G}(X)_{\mathbb{R}}$ because
$div(f_{\omega_{\gra}})=F_{\alpha}+\sum_{E\in \calN(X)}
v_{E}(\omega_{\gra})E$.  Thus the class of each $F_\gra$ belongs to $\sigma:=Eff(X)\cap Div^{G}(X)_{\mathbb{R}}$. This will implies that $Eff(X)$ is generated by the others $B$-stable prime divisors.

\begin{proof}[Proof of Theorem \ref{eff cone1}.]
Exactly as in \cite{Br07}, Lemma 2.3.1, we can prove that
$\sigma$ is $cone([E], \ E\in
\calN(X))$; thus $\sigma$ is simplicial because $\calN(X)$ is a basis of
$\widetilde{Div}^{G}(X)_{\mathbb{R}}$. Moreover,  $\sigma$ contains all the
$[F_{\alpha}]$ because
$div(f_{\omega_{\gra}})=F_{\alpha}+\sum_{E\in
\calN(X)}(\omega_{\gra},v_{E})E$ and the $v_{E}$ are
antidominant.
Therefore, the theorem is proved  if $\rho$ is injective.

In the
general case, $Pic(X)_{\mathbb{R}}=\bigoplus_{E\in
\calN(X)}\mathbb{R}[E]\oplus \bigoplus_{\alpha^{\vee}:
|\rho^{-1}(\alpha^{\vee})|=2}\mathbb{R}[F_{\alpha}^{+} ]$.
Indeed, $Pic(X)_{\mathbb{R}} $ $(=Cl(X)_{\mathbb{R}}$) is generated by the $[E]$ with  $E\in
\calN(X)$, the $[F_{\alpha} ]$ and the $[F_{\alpha}^{+} ]$ with
$|\rho^{-1}(\alpha^{\vee})|=2$ (see \S\ref{sez: N+Dg/h} for the definition of
 $F_{\alpha}$  and $F_{\alpha}^{+}$). Moreover, the
 relations are freely generated by the following ones $(*_{\gro_\gra})$: $[F_{\alpha}]=\sum_{E\in
\calN(X)}(\omega_{\gra},-v_{E})[E]$ with $\alpha\in
\overline{R}_{G,\theta}$. Indeed, $Cl(X)$ is generated by the $B$-stable prime divisors with relations $div(f_\gro)=0$ for any $\gro\in\chi(S)$; moreover the relation $(*_{\gro_\gra})$ is the one corresponding to the fundamental spherical weight $\gro_\gra$ (see also \S\ref{sez: picard}).
Observe that  $\rho$ is injective over
$\calF(X)$ because  $X$ is $\mathbb{Q}$-factorial.
Hence, $Eff(X)$ is generated by $\calN(X)$ and by $D(G/H)\senza D(G/H)^{H}$. Each $[F_\gra^+]$  generates an extremal ray of $Eff(X)$ because for any divisor $F_\gra^++div(f_\gro)=\sum_{E\in \calN(X)}n_EE+\sum_{\grb\in\rho(D(G/H)^H)}n_\grb F_\grb+\sum_{\grb\notin\rho(D(G/H)^H)}n_\grb^+F_\grb^++\sum_{\grb\notin\rho(D(G/H)^H)}n_\grb^-F_\grb^-$ such that $n_\gra^+=0$, we have $n_\gra^-=-1$.
We can argue similarly for the $F_\gra^-$.

%Let $J$ be the minimal set of   primitive generators of
%$Eff(X)$ (we say that $v\in Pic(X)\senza \{0\}$ is primitive if
%$\frac{1}{n}v\,\notin\, Pic(X)$ for any integer $n\neq\pm1$). For each
%coroot $\alpha_{i}^{\vee}\in \overline{R}_{G,\theta}^{\vee}\senza
%\rho(D(G/H)^{H})$, $J$ contains at least the class of a color  in
%$\rho^{-1}(\alpha_{i}^{\vee})$, say $[F_{\alpha_{i}}^{+}] $, because
%$dim\,Eff(X)=dim\,
%Pic(X)_{\mathbb{R}}$.
%But,  also $[F_{\alpha_{i}}^{-}]$ is contained in $J$ because
%$F_{\alpha_{i}}^{-} $ is effective and $F_{\alpha_{i}}^{-} =
%(F_{\alpha_{i}}^{-}+F_{\alpha_{i}}^{+} ) -F_{\alpha_{i}}^{+} $ with
%$[F_{\alpha_{i}}^{-}+F_{\alpha_{i}}^{+} ]\in \sigma$.

If the class of  $E\in \calN(X)$ does not generate an extremal ray of $Eff(X)$, then $[E]$ is a
positive combination of the classes of the other $G$-stable prime
divisors and of the [$F_{\alpha}]$ with $\alpha^{\vee}\in
\overline{R}_{G,\theta}^\cech\senza \rho(D(G/H)^{H})$. But $[E]$
generates an extremal ray of $\sigma$, so it has to be proportional
to some $[F_{\alpha}]$.  \end{proof}

Given any
symmetric variety $X$, let $X^{\leq 1}$ be the open $G$-subvariety
composed of orbits of codimension at most 1. We have an equivariant
morphism $q:X^{\leq
1}\rightarrow X_{0}^{sph}$ which can be extended to
$X$ if and only if $X$ is toroidal. Here $X_{0}^{sph}$ is the
wonderful completion of $G/\overline{H}^{sph}$.

\begin{thm}\label{eff cone2}
Let $X$ be a $\mathbb{Q}$-factorial complete symmetric variety. Then:
\begin{enumerate}
\item The class of a $G$-stable prime divisor $E$ belongs to the
cone generated by the classes of colors if and only if  it is proportional to
some $[F_{\alpha}]$.
\item The class of a $G$-stable prime divisor $E$ is proportional to
$[F_{\alpha}]$ if and only if the irreducible factor $R$
of $R_{G,\theta}$  containing $\alpha$ is
orthogonal to $v_{E'}$ for any $E'\in \calN(X)$ different from $E$.
\item The class of a $G$-stable prime divisor $E$ is proportional to some
$[F_{\alpha}]$ if and only if there is a $G$-equivariant morphism
$\varphi:X^{\leq1}\rightarrow X'$ such that $X'$ is a wonderful
symmetric $G$-variety, $\varphi(E)\subsetneq X'$ and
$\varphi(E')=X'$  for each $G$-stable prime divisor $E'$ of $X$
different from $E$.
\item If  such a morphism exist, we can identify the restricted root system of
$X'$ with a product $R$ of irreducible factors of $R_{G,\theta}$.
Then $[E]$ is proportional to $[F_{\alpha}]$ for each simple root
$\alpha$ in $R$. We can also suppose that the stabilizer of
$\varphi(H/H)$ is generated by $\overline{H}^{sph}$ and some normal,
$\theta$-stable, connected subgroup of $G$. If moreover $X$ is
toroidal, then $\varphi$ can be extended to $X$,
$-v_{E}$ is a fundamental spherical weight and $X'$ is a product of wonderful symmetric varieties of rank one.
\item  $[$Lemma 2.3.2 of \cite{Br07}$]$ The class of a $G$-stable prime divisor  $E$ generates an extremal ray of $Eff(X)_\mR$ which does not contain the class of a color if and only if $dim\,H^0(X,\calO(mE))=1$ for each positive integer $m$.
\item $[$Lemma 2.3.3  of \cite{Br07}$]$ The class of a color   $F$ generates an extremal ray of $Eff(X)_\mR$ which does not contain the class of any $E\in \calN(X)$ if and only if there is a $G$-equivariant morphism $\varphi:X\rightarrow G/P$, where $P\supseteq H$ is a maximal parabolic subgroup, such that $F$ is the preimage of the Schubert divisor (the unique $B$-stable prime divisor) in $G/P$.
\end{enumerate}
\end{thm}

Before to prove the theorem, we do some remarks.

\begin{rem}\label{rem: commento eff cone 2}
Observe that given an irreducible factor $R$ of $R_{G,\grt}$, there is always an $E\in
\calN(X)$ with $(v_{E},R)\neq 0$ because of the completeness
of $X$. The statement of the third point is very similar to that  of
Lemmas  2.3.4 in
\cite{Br07}.
\end{rem}

\begin{rem}\label{rem: canditato phi} Write $R_{G,\grt}$ as a product $\prod R_i$ of irreducible factors, then $X_0^{sph}$ is a product $\prod X_{R_i}$ by Corollary \ref{compat spherical closure}. In the proof we show that to check if $[E]$ belongs to $\mR^{\geq0}[F_\gra]$ with $\gra\in R_{i_0}$ is sufficient to check if $\pi_{i_0}\circ q:X^{\leq1}\rightarrow X_{R_0}$ satisfies the conditions of  point (3).
\end{rem}

\begin{rem} In the setting of wonderful varieties, there are never two colors whose classes are proportional.
\end{rem}

\begin{rem}\label{rem: colori estremali}
In \cite{Br07} the Lemma 2.3.1 and 2.3.2 are stated for wonderful varieties, but their proof holds for any $\mathbb{Q}$-factorial complete spherical variety whose open orbit is sober, i.e. $N_G(H)/H$ is finite. In the setting of symmetric varieties, one can explicitly  construct the morphism of point (6). Let $F$ be a color as in the point (6) and let   $\alpha^\cech$ be  $\rho(F)$. Then, by Theorem \ref{eff cone1},  $\rho^{-1}(\alpha^\cech)$ contains two colors and $\gro_\gra$ is  the sum $\gro_1+\gro_2$ of two (possibly equal) fundamental weights of $G$. We can also suppose that $G^\grt=P(\gro_1)\cap P(-\gro_1)$ and $\gro_2=-\varpi_o\gro_1$, where $\varpi_o$ is the longest element of $W_G$. Also the other color in $\rho^{-1}(\alpha^\cech)$ satisfies the conditions of (6) and the corresponding applications are the following:
\[ G/G^\grt\rightarrow G/P(\gro_1)\subset \mP(V(\gro_1))\]
\[ g\rightarrow g\cdot v_{\gro_1}\]
and
\[ G/G^\grt\rightarrow G/P(-\gro_1)\subset \mP(V(\gro_2))\]
\[ g\rightarrow g\cdot v_{-\gro_1}\]
where $v_\chi$ is a weight vector of weight $\chi$. Furthermore, $G/P(-\gro_1)$ is isomorphic to $G/P(\gro_2)$. There are some difference according to whether $\gra$ is exceptional or not.
If $\gra$ is exceptional then   $\gro_1$ is different from $\gro_2$; in particular the stabilizer of $F_\gra^+$ is different from the stabilizer of $F_\gra^-$. Moreover $\mP(V(\gro_1))$, resp. $\mP(V(\gro_2))$, contains a unique point fixed by $G^\grt$.

Instead, if  $\gra$ is non-exceptional then    $\gro_1=\gro_2$; in particular,  the stabilizer of $F_\gra^+$ is equal to the stabilizer of $F_\gra^-$. Moreover,  $\mP(V(\gro_1))$ contains two points fixed by $G^\grt$, namely  $v_{\gro_1}$ and $v_{\gro_1}$.
In this case there is an element $n$ of $N_G(H)\senza H$ which exchanges $F_\gra^+$ with $F_\gra^-$; moreover $n$ exchanges $v_{\gro_1}$ with $v_{-\gro_1}$.
\end{rem}

\begin{proof}[Proof of Theorem \ref{eff cone2}.] We have already showed the first point in the proof of the
previous theorem. First, we will prove the point (2). Then we will  use it to prove the points (3) and (4). We will use also the Corollary \ref{compat spherical closure} to find an explicit candidate for the application $\varphi$ (see also Remark \ref{rem: canditato phi}).  Because
$div(\omega_{\gra})=F_{\alpha}+$ $\sum_{E\in \calN(X)}$ $
(\omega_{\gra},v_{E})E$,  $[E]$ is proportional to $[F_\alpha]$
if and only if $(\omega_{\gra},v_{E'})=0$ for any $E'\neq E$. But the
$-v_{E'}$ are dominant coweights. Thus,   if
$[E]\in \mR^{\geq0}[F_{\alpha}]$ and $\gra'$ belongs to the  irreducible factor of
$R_{G,\theta}$ containing $\alpha$, then
$[E]\in \mR^{\geq0}[F_{\alpha'}]$.

Given an irreducible factor  $R$ of $R_{G,\grt}$,  we
can write, by Corollary \ref{compat spherical closure}, $X_{0}^{sph}=X_{1}\times X_{2}$ where the $X_{i}$ are
wonderful varieties and the restricted root system of $X_{2}$ is
equal to $R$. Given any $v\in C^{-}$ and $\gra\in R$, $(\omega_{\gra},v)=
0$ if and only if, for any $\alpha'\in R$, $-\omega_{\gra'}^{\vee}$
is not contained in the face of $C^{-}$ whose relative interior
contains $v$. Thus, $[E]$ is proportional to $[F_{\alpha}]$ if
and only if the following condition $(*)$ holds: if $K\in
\calN(X^{sph}_{0})$ contains the image of a $G$-stable prime divisor of
$X^{\leq1}$ different from $E$, then $K$ has the form $K'\times
X_{2}$ with $K'\in \calN(X_{1})$. Hence, if $[E]$ is proportional to
$[F_{\alpha}]$, then the projection on $X_{2}$ of any $q(E')$ with $E'\neq E$  is the
whole $X_{2}$.

Vice versa suppose that exists a morphism $\varphi$ as in the
statement. By the Corollary \ref{spherical closure} and by the
description of morphisms between spherical varieties, we can suppose
that the stabilizer of $\varphi(H/H)$ is  generated by
$\overline{H}^{sp}$ and some normal, $\theta$-stable connected
subgroup of $G$ (we may have to compose or to lift $\varphi$ with a
finite equivariant morphism). We need also the following property:
if $H'$ is a symmetric subgroup of $G$ which contains $H$ and does
not contain any normal connected subgroup of $G$, then $H'\subset
N_{G}(H)$ (see \cite{dCP1}, Lemma 1.7). Therefore, $\varphi$ is the
composite of $q:X^{\leq
1}\rightarrow X_{0}^{sph}$ with a projection $X_{0}^{sph}=X_{1}\times
X_{2}\rightarrow X_{2}$. Now, the hypotheses on $\varphi$ implies
the condition $(*)$.

Finally, if $X$ is toroidal then $-\omega_{\gra}^{\vee}$ is contained
in $\calN(X)$ for each $\gra\in\ol{R}_{G,\grt}$, so $[F_{\alpha}]$ can be proportional only to
$E_{-\omega_{\gra}^{\vee}}$. \end{proof}

\begin{rem}\label{rem: riduzione a G stable} Let $X$ be any projective symmetric variety (possibly non $\mathbb{Q}  $-factorial)
and let  $D=$ $\sum_{F\in D(G/H)}$ $ a_{F}F+ \sum_{E\in \calN(X)}b_{E}E$ be an effective  Cartier divisor on $X$, so $a_{F},b_{E}\geq0$. Up to exchanging some $F_{\alpha_{i}}^{+}$ with $F_{\alpha_{i}}^{-}$, there
is an effective divisor $D'=D_1+D_2$  linearly equivalent to $D$ and such that: i) $D_{1}$
is $G$-stable and effective;
ii) $D_{2}=\sum a^{+}_{i}F_{\alpha_{i}}^{+}$ with $a^+_{i}\geq0 $ for each $i$.
Moreover, $D$ is nef (resp. big) if and  only if $D_{1}$ is  nef (resp. big). Indeed, we can suppose $X$ toroidal by taking the pullback of these line bundles to a desingularization $X'$ of $X^{dec}$. Then  $D$ is  nef (resp. big) if and  only if $D'$ is   nef (resp. big). Moreover,
$h^{D'}=h^{D_{1}}$ and the coefficients of $D'$ with respect to the $F_{\alpha} ^{-}$ are all zero (and lesser than the coefficients of $D'$ with respect to the $F_{\alpha} ^{+}$). Finally, if $D'=\sum_{F\in D(G/H)^{H}}\, c_{F}F
+\sum d_{\gra}^{+}F_{\alpha}^{+}+ \sum d_{\gra}^{-}F_{\alpha}^{-}+\sum_{E\in \calN(X)} f_{E}E$, the dimension of $H^{0}(X,\mathcal{O}(D'))$ and the combinatorial conditions on the nefness of $D_{1}$ depend only on the linear functions $h^{D_{1}}_\calC$,  on the $c_{F}(=0)$ and on the $min\{d_{\gra}^{+},d_{\gra}^{-}\}(=0)$.
\end{rem}

\begin{cor}\label{eff cone of singular X} Let $X$ be a projective symmetric variety (possibly non $\mathbb{Q}  $-factorial).
Then $Eff(X)$ is the intersection of $Pic(X)_{\mathbb{R}}$ with the
polyhedral cone of $Cl(X)_{\mathbb{R}}$ whose extremal rays are
generated by the classes of the colors not in $D(G/H)^{H}$ and by
the classes of the $G$-stable prime divisors which are not linearly
equivalent to a multiple of $[F_{\alpha}]$ with
$F_{\alpha}\notin\, D(G/H)^{H}$.   Moreover, the statement of
Theorem \ref{eff cone2}  holds   again.
\end{cor}

\begin{proof}[Proof of Corollary \ref{eff cone of singular X}.] To prove the corollary it is sufficient to show that
there is a $\mathbb{Q}$-factorial complete symmetric variety $X'$
and an equivariant morphism $\psi: X'\rightarrow X$ which induces an
isomorphism between $(X')^{\leq1}$ and $X^{\leq1}$; in particular
$\psi$ induces an isomorphism between $Cl(X)$ and $Cl(X')$. Moreover
$\psi^{*}Eff(X)=Eff(X')\cap \psi^{*}Pic(X)_{\mathbb{R}}\subset
Cl(X)_{\mathbb{R}}$
and $Pic(X)_{\mathbb{R}}\cong \psi^{*}Pic(X)_{\mathbb{R}}\subset Pic(X')_{\mathbb{R}}=Cl(X')_{\mathbb{R}}\cong Cl(X)_{\mathbb{R}}$.
Observe that the Theorem \ref{eff cone2} depends only on
$X^{\leq1}$.

Now, we will construct $X'$. The procedure will be more complicated if $X$ is neither non-exceptional nor toroidal. We need to define a new variety $\wt{X}$, isomorphic to $X$ in codimension 1: let $\mbF(\wt{X}):=\{(\calC,\wt{\calF}): \  (\calC,\calF)\in\mbF(X)\}$, where $\wt{\calF}:=\rho(\rho^{-1}(\calF))$, and let $\wt{X}$ be the corresponding variety. Remark that if $X$ is  non-exceptional or toroidal, then it is equal to $\wt{X}$.
We have a morphism $p:X\rightarrow \wt{X} $ which is an isomorphism between  $X^{\leq1}$ and  $\wt{X}^{\leq1}$, thus it is sufficient to find a variety $\varphi:X'\rightarrow \wt{X}$ over $\wt{X}$ such that: 1) $\varphi$ factorizes by $p$ and 2) $\grf$ is an isomorphism in codimension 1.

First, we define the fan $\mbF^f(X')$ associated to $X'$. The idea is the following: the cones in $\mbF^f(\wt{X})(=\mbF^f(X))$ are generated by some faces of an appropriate polytope in $\chi_*(S)$ (which is the polar polytope of the moment polytope of  an ample bundle $D$ over $\wt{X}$); we triangularize the faces of such polytope and define $\mbF^f(X')$ as the cones generated by the simplices obtained from the previous faces.

To define the previous polytope we need an ample Cartier divisor  $D$  over $\wt{X}$ such that: i) the interior of  $P(D)$ contains 0 and ii) $D$ is linearly equivalent to a $G$-stable divisor. We have defined $\wt{X}$ to assure the existence of such a divisor. Now, we will find it;  let $D'$ an ample Cartier divisor on $X$. As in the Remark \ref{rem: riduzione a G stable} we can write $D'=D_{1}+D_{2}+D_{3}$ where 1) $D_{1}$ is linearly equivalent to a  $G$-stable divisor; 2) $D_{2}+D_{3}$ is a positive linear combination of the $F_{\alpha}^+$. Moreover, we can suppose that $D_{2}$ is $\sum_{\gra\in I} a_{\alpha}^+F_\gra^+$,  where
$I$ is the set of $\gra$ such that  $F_{\alpha}^+$ belongs to  $\calF(X)$ (and $D_3$ is $\sum_{\gra\notin I} a_{\alpha}F_\gra^+$). \vspace{0.3 mm} One can easily show that $D'':=D_{1}+D_2+\sum_{\gra\in I} a_{\alpha}^+F_\gra^-$ is an ample divisor over $\wt{X}$. Indeed, $h^{\wt{X},D'}=h^{X,D'}=h^{X,D_1+D_2}$ and the minimum of the coefficients of $D''$ w.r.t. the colors in $\rho^{-1}(\gra^\cech)$ is lesser than the corresponding minimum for $D'$. Remark that we have used the fact that $X^{\leq1}$ is isomorphic to $\wt{X}^{\leq1}$, so $Cl(X)\cong Cl(\wt{X})$ (but $Pic(X)$ can be non-isomorphic to $Pic(\wt{X})$). Then $D^{(3)}:=D''+\sum_{\gra\notin \rho(\calF(\wt{X}))}F_\gra$ is ample over $\wt{X}$ and is linearly equivalent to a $G$-stable effective divisor $D^{(4)}$ such that $h^{D^{(4)}}(v_E)>0$ for each $E\in \calN(\wt{X})$.  Indeed,  none irreducible factor of $R^\cech_{G;\grt}$ can be contained in $span\{\rho(\calF(X))\}$ because the colored cones are strictly convex.  Therefore we can choose $D$ as $n D^{(4)}+\sum_{\gra\notin \rho(\calF(\wt{X}))}F_\gra+\sum_{\gra\in \rho(\calF(\wt{X})) }(F_\gra-\sum_{\calN(\wt{X})}v_E(\gro_\gra)E)$ with $n>>0$.

Let $P$ be the polar polytope of $P(D)$, i.e. $\{n\in \chi_{*}(S)_{\mathbb{R}}: m(n)\geq-1 \ \forall m\in P(D)\}$.
Then, given any cone $\calC$ in the
fan $\mbF^f(\wt{X})$ of $\widetilde{X}$, $\calC$ is generated by an appropriate face of $P$. Observe
that there are faces of $P$ which are associated to none colored
cone of $\widetilde{X}$.  Let $A'$ be the set of vertices of $P$ and set
$A=A'\cup\{0\}$. We need  to give a triangulation of $P$ with vertices in $A$.
Given a polytope $Q$ generated by the points $S=\{q_{1},...,q_{m}\}$, a \textit{subdivision} of $Q$ with vertices in $S$ is a finite collection $\{Q_{1},...,Q_{r}\}$ of polytopes such that: i)
$Q$ is the  union  $\bigcup Q_{i}$; ii) the
vertices of each $Q_{i}$ are drawn from $S$; iii) if $i\neq j$ then
$Q_{i}\cap Q_{j}$ is a common (possibly empty) face of the
boundaries of $Q_{i}$ and $Q_{j}$.  If all the $Q_{i}$ are simplices, the subdivision is called a \textit{triangulation}.  Before to define the desired triangulation of $P$, we
need to define an elementary construction step.

Let $Q$ be a $n$-dimensional polytope in $\mathbb{R}^{n}$, let $F$ be a $(n-1)$-dimensional face of
$Q$, let $H$ be the unique
hyperplane  containing $F$ and let $v$ be a point in $\mathbb{R}^{n}$.  The polytope $Q$ is contained in
exactly one of the closed halfspaces determined by $H$. If $v$ is
contained in the opposite open halfspace, then $F$ is said to be
\textit{visible} from $v$. If $Q$ is a $k$-dimensional polytope in $\mathbb{R}^{n}$ with
$k<n$ and $v\in Aff(Q)$, then the above definition can be  modified in
the obvious way, so that everything is considered relative to the
ambient space $Aff(Q)$. Suppose $S=\{Q_{1},...,Q_{m}\}$ is a
subdivision of a $n$-dimensional polytope $Q=convex(V)$ in
$\mathbb{R}^{n}$ and let $v\in V$. The result of \textit{pushing} $v$ is, by definition, the
subdivision $S'$ of $Q$ obtained by modifying the $Q_{i}\in S$ as
follows:

\begin{itemize}
\item If $v\notin\, Q_{i}$, then $Q_{i}\in S'$.
\item If $v\in Q_{i}$ and $convex(vert(Q_{i})\senza \{v\})$ is
$(n-1)$-dimensional (i.e. $Q_{i}$ is a pyramid with apex $v$), then
$Q_{i}\in S'$.
\item If $v\in Q_{i}$ and $Q_{i}':=convex(vert(Q_{i})\senza \{v\})$ is
$n$-dimensional, then $Q'_{i}\in S'$. Also, if $F$ is any
$(n-1)$-dimensional face of $Q_{i}'$ that is visible from $v$, then
$convex(F\cup\{v\})\in S'$.

\end{itemize}

Let $Q=convex(V)$ and order the point of $V=\{v_{1},...,v_{m}\}$ in an arbitrary way; then
the subdivision obtained by starting with the trivial one and
pushing the points of $V$ in that order is a triangulation (see \cite{Le97},
\S14.2). Returning to our problem, let $P$ and $A$ as the first part of the proof and
order the points of $A$ so that 0 is the first point.
Let $T$ be the triangulation
of $P$ obtained from the trivial subdivision by  pushing the points of $A$ in the chosen order. This triangulation induces a triangulation
of the (proper) faces of $P$. Let $\{T_{i}\}_{i\in I}$ be the set of $(s-1)$-dimensional simplices
obtained in such way, where $s=rank\,\chi_*(S)$, and let $I'\subset I$ be the family of simplices
whose relative interior intersects $C^{-}$. Given $i\in I'$, let $\calC_{i}$ be the cone generated by $T_{i}$.
We want to define $\mbF(X')$ so that $\mbF^f(X')$ is composed by the faces of all the $\calC_{i}$ with $i\in I'$. Such set is a fan by the definition of subdivision of a polytope; moreover its support is the same of the one of $\mbF(\wt{X})$

For each $i\in I'$, we define $\calF_{i}$ as follows. For each
$\alpha^{\vee}$ in $T_{i}$,  choose a color $E_{\alpha}$ in $\calF\cap
\rho^{-1}(\alpha^{\vee})$, where   $(\calC,\calF)$ is  the ($s$-dimensional) colored cone of $X$ containing $\calC_{i}$. Finally, define $\calF_{i}$ as the
set of such $E_{\alpha}$. If  $\alpha_{j}^{\vee}$ is contained in two different simplices, say
$T_{i}$ and $T_{j}$,
then we choose the same  $E_{\alpha}$ for both $\calF_{i}$ and $\calF_{j}$. Remark that we need to work with the colored cones of $X$ because we want that $X'$ dominates not only $\wt{X}$, but also $X$. The previous choices are possible because of the combinatorial conditions for the ampleness of a Cartier divisor $D'$ on $X$. Indeed, suppose by contradiction that there are a simple restricted root $\alpha$ and two colored cones of $X$, say $(\calC_{1},\calF_{1})$,  $(\calC_{2},\calF_{2})$,  such that $F_{\alpha}^{+}\in \calF_{1}\senza \calF_{2}$ and $F_{\alpha}^{-}\in \calF_{2}\senza \calF_{1}$. Write $D'=\sum_{D(G/H)^{H}} a_{F}F
+\sum b_{\grb}^{+}F_{\grb}^{+}+ \sum b_{\grb}^{-}F_{\grb}^{-}+\sum_{ \calN(X)} c_{E}E$  and let $h$ be the convex function associated to $D'$. Then $b_{\gra}^{+}=b_{\gra}^{-}=h(\alpha^{\vee})$ because $F_{\alpha}^{+}, F_{\alpha}^{-}\in \calF(X)$. Moreover $b_{\gra}^{+}=h_{\calC_{1}}(\rho(F_{\alpha}^{+}))=h_{\calC_{1}}(\rho(F_{\alpha}^{-}))<h(\rho(F_{\alpha}^{-}))=h(\alpha^{\vee})$, a contradiction.
Thus $\{(\calC_{i},\calF_{i}):
i\in I'\}$ is a colored fan and the associated symmetric variety
satisfies the requested properties. \end{proof}

\section{Bigness of   $\mathbb{Q}$-divisors on a complete symmetric
variety}\label{sez: bigness}

First we describe the big cone of any $\mQ$-factorial, projective symmetric variety. Then we will prove two criterions for a nef ($G$-stable) divisor to be big.

\subsection{The big cone}\label{sez: big cone}

\begin{thm}\label{bigg cone}
Let $X$ be a projective, $\mQ$-factorial symmetric variety. Then $Big(X)$ is the union of the following cones (whose closure is simplicial):  $\bigoplus_{E\in \calN(X)}\mR^{>0}[E]\oplus \bigoplus_{\gra^\cech \notin \rho(D(G/H)^H)}\mR^{\geq0} [F_\gra^\bullet]$, where the $F_{\gra}^\bullet\in\rho^{-1}(\gra^\cech)$ are chosen in all the ways possible.
\end{thm}

To prove such theorem we will use Theorem \ref{eff cone1} plus the explicit expression of the relations $(*_{\gro_\gra})$.

\begin{proof} Let $I$ be $\ol{R}_{G,\grt}^\cech\senza\rho(D(G/H)^H)$. First, we prove that all the cones in the statement are contained in $Big(X)$. It is sufficient to prove that $\dot{\grs}:=\bigoplus_{E\in \calN(X)}\mQ^{>0}[E]$ is contained in $Big(X)$ because the sum of a big divisor with an effective one is big. Given any element $[D]$ of $\dot{\grs}$,  there are $r_{\alpha}$ such that $[D']:=[D]-\sum_I r_\gra [F_\gra]$ belongs to $\dot{\grs}$, because $\dot{\grs}$ is open and all the $[F_\gra]$ belongs to the closure $\grs$ of $\dot{\grs}$. Thus $[D]=[D']+\sum_I r_\gra [F_\gra]$ is big by Theorem \ref{eff cone1}.

Now, let $D$ be a big divisor. Up to exchanging some $F_\gra^+$ with the corresponding $F_\gra^-$, we can write $D=\sum_{E\in\calN(X)}n_E E+\sum_{\gra\in\ol{R}_{G,\grt}}n_\gra F_\gra+\sum_{\gra\in I}n_\gra^+F_\gra^+$ with positive coefficients. This divisor is linear equivalent to an effective divisor $D'=\sum_{E\in\calN(X)}m_E E+\sum_{\gra\in I}m_\gra^+F_\gra^+$, so it is sufficient to show that all the $m_E$ are strictly positive. There are two cases: i)  the class of any $G$-stable prime divisor generates an extremal ray of $Eff(X)$; ii) there is some $E$ linearly equivalent to a multiple of $F_\gra$ with $\gra\in I$ (see Theorem \ref{eff cone1}). In the first case, all the $m_E$ (and all the $m_\gra^+$) are strictly positive because $Big(X)$ is an open cone of dimension equal to $|\calN(X)|+|I|$ and all the vectors in the sum generate an extremal ray of $Eff(X)$. In the second case, some $[E]$ is equal to some $t[F_\gra^+]+t[F_\gra^-]$, so we can't use the same argument. Let $J$ be the set of $G$-stable prime divisors which generate an extremal ray of $Eff(X)$ and let $K\subset I$ be subset of   roots such that $[F_\gra]$ is proportional to some $[E]$ with $E\in \calN(X)$. Then we can write $[D]=\sum_{E\in J}m_E [E]+\sum_{I\senza K}m_\gra^+[F_\gra^+]+\sum_{  K}r_\gra^+[F_\gra^+]+\sum_{ K}r_\gra^-[F_\gra^-]$ and, as before, the coefficients must be strictly positive. Given $E'\notin J$ and $\gra\in K$ such that $[E']=t_\gra[F_\gra]$ with $t_\gra>0$, we have $r_\gra^+=t_\gra m_{E'}+ m_\gra^+$ and $r_\gra^-=t_\gra m_\gra^-$. Therefore $m_E'$ is strictly positive  for all $E'$.
\end{proof}

\subsection{Bigness of nef divisors}\label{sez: bigness + nef}

Now, we want to study the bigness of a fixed (nef) Cartier divisor.

\begin{rem}\label{rem: rid a toroid per big} We want to observe that to study the bigness of a fixed (Cartier)
$\mathbb{Q}$-divisor we can reduce to the case of a smooth toroidal
symmetric variety with $H$ wonderful. These are the smooth symmetric varieties proper  over a wonderful one.

First of all, we can reduce to the smooth toroidal case because of the Remark \ref{rem: sezioni pullback}.
%given an equivariant morphism $\pi:X'\rightarrow X$ of normal complete
%embeddings of $G/H$ (which extends  the identity morphism of $G/H$) and
%a line bundle $L$ over $X$, then $H^{0}(X,L)$ is isomorphic to
%$H^{0}(X',\pi^{*}L)$, thus we can reduce ourselves  to the  toroidal case.
Suppose now $X$ toroidal and let $X ^{sph}$ be the completion
of $G/\overline{H}^{sph}$ with the same colored fan of $X$ (see Remark \ref{rem: fan spherical closure}). Then $X ^{sph}$ dominates the wonderful
completion $X ^{sph}_{0}$ of $G/\overline{H}^{sph}$. We have unique
equivariant morphisms $\phi:X\rightarrow X ^{sph}$ and $X_{0}\rightarrow
X ^{sph}_{0}$ (which send $H/H$ to
$\overline{H}^{sph}/\overline{H}^{sph}$) and the pushforwards of
such morphisms are isomorphisms between their rational Picard
groups. In general the pushforwards are defined between the
(rational) class groups; in our case the pushforward define an isomorphism between the rational class groups which restricts to an isomorphism between the rational Picard groups.

Indeed, $dim\, Cl(X)_\mQ=dim\, Cl(X^{sph})_\mQ$ and $dim\, Pic(X)_\mQ=dim\, Pic(X^{sph})_\mQ$ because $D(G/\overline{H}^{sph})\equiv D(G/H)$
and  $X^{sph}$ \lq\lq has" the same colored fan of $X$ (forgetting the
lattice in $(\mathbb{C}(G/H)^{B}/\mathbb{C}^{*})_{\mathbb{Q}}$); in other words the rational Picard group does not depend on the lattice $\chi_*(S)$.
If $X^{sph}$ is smooth, then $\phi_{*}\circ \phi^*:Cl(X^{sph})_\mQ\rightarrow Cl(X^{sph})_\mQ$ is $(deg\,\phi)Id$ and $\phi^*(Pic(X^{sph})_\mQ)\subset Pic(X)_\mQ$, so the claim holds. In the general case we take a desingularization $\psi:\ol{X}^{sph}\rightarrow X^{sph}$ of $X^{sph}$ and  we define
$\ol{X}$ as the completion of $G/H$ with colored fan $\mbF(\ol{X}^{sph})$,  $\varphi:\ol{X}\rightarrow X$, resp. $\ol{\phi}:\ol{X}\rightarrow \ol{X}^{sph}$, as the obvious maps.  Then we regard to the subspace $\psi^*(Pic(X^{sph})_\mQ)$ of $Pic(\ol{X}^{sph})_\mQ$ (isomorphic to $Pic(X^{sph})_\mQ$) and use the following facts: i) $(\ol{\phi}^{\ *}\comp \psi^*)(Pic(X^{sph})_\mQ)=(\varphi^*\comp \phi^*)(Pic(X^{sph})_\mQ)$ is contained in $\varphi^*(Pic(X)_\mQ)$ and ii) $\phi_*\comp \varphi_*=\psi_*\comp \ol{\phi}_*$.

A line bundle $\mathcal{O}(D) $ on $X$ is big if and only if $\phi_{*}(\mathcal{O}(D))$ is big. Indeed if $\chi_{*}(T/T\cap \overline{H}^{sph})\subset \frac{1}{m} \chi_{*}(T/T\cap H)$, then $|P(rD)\cap \chi_{*}(T/T\cap H)|\leq| P(rD)\cap \chi_{*}(T/T\cap \overline{H}^{sph})|\leq |P(mrD)\cap \chi_{*}(T/T\cap H)|$ for each positive integer $r$ (see \S1.6
and   \cite{Br89}, \S3). The last inequality holds because the multiplication by $m$ defines an inclusion of  $P(rD)\cap \chi_{*}(T/T\cap \overline{H}^{sph})$ in  $P(mrD)\cap m\chi_{*}(T/T\cap \overline{H}^{sph})$ ($\subset P(mrD)\cap \chi_{*}(T/T\cap H)$).
\end{rem}

Let $X$ be a toroidal symmetric variety and suppose $X_{0}$ smooth.
Using the results of  \cite{Br89}, one can show that  a $G$-stable divisor  $D$ on $X$ is ample
(resp. nef) if and only if $\mathcal{O}(D)|Z^{c}$ is ample (resp. nef). Moreover,
this holds  if and only if the restrictions of $\mathcal{O}(D)$ to $Z$ and to the
closed orbits are ample (resp. nef). These last conditions can be
stated as  appropriate conditions on the function $h=\{h_{\calC}\}$
and on the weights $h_{\calC}$, where the $(\calC,\vuoto)$ are the maximal colored cones of $\mbF(X)$: 1) $D$ is nef if and  only if $h$ is convex and the $-h_{\calC}$ are spherical weights; 2)   $D$ is ample if and  only if $h$ is strictly convex and the $-h_{\calC}$ are regular spherical weights (i.e. they are  strictly dominant weights of $R_{G,\theta}$). We want to prove a similar condition for the bigness of any nef line bundle.
Observe that, given a closed $G$-orbit $\mathcal{O}_{\calC}$ of $X$ associated to a (maximal) colored cone $(\calC,\vuoto)$, the weight of the fiber of $\mathcal{O}(D)$ over the $B$-stable point of $\mathcal{O}_{\calC}$ is $-h_{\calC}$ because $D$ is $G$-stable (see \cite{Bi90}, \S2).
Suppose by simplicity $X$ toroidal. First we prove that a nef $G$-stable divisor is big if and only if its restriction to the associate complete toric variety $Z^c$ is big (see Proposition \ref{bigness of toroidal}). Then, we use the fact the such restriction is big if and only if $vol(Q(D))$ is strictly positive (see \S\ref{sez: picard} for the definition of $Q(D)$). For example, if $X$ is wonderful, $\grt$ is indecomposable and $D$ is a $G$-stable divisor with associated function $h$ one can easily prove the following description (recall that there is a unique  maximal colored cone  $(\calC,\vuoto)$):
\begin{itemize}
\item $D$ is ample if and only if $-h_{\calC}$ is spherical and regular (i.e. it is an strongly dominant weight of $R_{G,\grt}$);
\item $D$ is nef if and only if $-h_{\calC}$ is spherical;
\item $D$ is big and nef if and only if $-h_{\calC} $ is spherical and non-zero.
\end{itemize}
When $X$ is only toroidal, but  $\grt$ is again indecomposable, we will prove that $D$ is nef and big if and only the sum $-\sum_{(\calC,\vuoto)\in \mF(X)\ maximal} h_\calC$ is spherical and non-zero (see Theorem \ref{bigness e nef}).

\begin{rem}\label{rem: decolor e' proiettiva} Given a complete symmetric variety  $X$,   let $p:X^{dec}\rightarrow X$ be the decoloration of $X$ and
let $Z^c$ be the complete toric variety associate to $X^{dec}$. If $X$ is projective then $X^{dec}$ and $Z^c$ are projective. Indeed let
$D$ be an ample divisor on $X$, then $D':=p^*D+\sum_{F\in D(G/H)}F$ is ample on $X^{dec}$. Indeed, $h^{p^*D}=h^{D'}$ is equal to the restriction of $h^D$ on $C^-$. Thus, $(h^{D'}_\calC) (F)\leq a_{F}<a_F+1$ for each $(\calC,\vuoto)\in \mbF(X^{dec})$ and  $F\in D(G/H)$ (here $a_F$ is the coefficient of $D$ with respect to $F$; see also \S\ref{sez: picard}).
\end{rem}

\begin{prop}\label{bigness of toroidal}
Let $X$  be a  projective symmetric variety and let $D$
be a $B$-stable, Cartier $\mathbb{Q}$-divisor on $X$. If $D$ is big
then the restriction $\mathcal{O}(p^*D)|Z^{c}$ to the associated complete toric
variety $Z^{c}$  is big. Moreover, if $D$ is $G$-stable and
$\mathcal{O}(p^*D)|Z^{c}$ is big then $D$ is big.
\end{prop}

\begin{proof}  By the previous discussion we can suppose that  $X$ is smooth and toroidal (see also the proof of the Corollary \ref{eff
cone of singular X}).  First, we describe the restriction $i^*Pic(X)_\mQ\rightarrow Pic(Z^c)_\mQ$, then we use the  Proposition \ref{cns D big} applied respectively to $X$ and $Z^c$. Under the previous assumption, we can define a linear map $i^{*}:Div^{B}(X)_{\mathbb{Q}}\rightarrow
Div^{T}(Z^{c})_{\mathbb{Q}}$ such that $[i^{*}(D)]=i^{*}([D])$ in
the following way: $ i^{*}: Div^{B}(X)_\mQ\twoheadrightarrow
Pic(X)_{\mathbb{Q}}\equiv Pic^{G}(X)_{\mathbb{Q}}\rightarrow
Pic^{T}(Z^{c})_{\mathbb{Q}}\equiv  Div^{T}(Z^{c})_{\mathbb{Q}}.$
Here $i:Z^{c}\hookrightarrow X$ is the inclusion and $Pic^{G}(X)$ is
the group of $G$-linearized line bundles. See \cite{Od88}, Proposition 2.1
for the last isomorphism. First we want to study the kernel and the image of $i^{*}$ by using some techniques similar to ones in \cite{V90} and in \cite{Bi90}.

Given a $G$-stable prime divisor $E$ on $X$ we define $E|Z$ as the
closure of $E\cap Z$ in $Z^{c}$, so $E|Z^{c}:=i^{*}(E)$ is $\sum_{w\in
W_{G,\theta}}wE|Z$ and has support $E \cap Z^{c}$. Hence, $E|Z^{c}$
and $E|Z$ are effective divisors on $Z^{c}$. Moreover, the
$T$-stable prime divisors on $Z^{c}$ are the $wE|Z$ with $E\in \calN(X)$
and $w\in W_{G,\theta}$ (actually $w$ is a fixed representant in
$N_{H^{0}}(T^{1})$ of the corresponding element in $W_{G,\theta}$). Let
$\pi:G\rightarrow G/H$ be the projection. The kernel of
$i^{*}:Pic(X)_{\mathbb{Q}}\rightarrow Pic(Z^{c})_{\mathbb{Q}}$ is
generated by the $[F_{\alpha}^{+}-F_{\alpha}^{-}]$ with
$\alpha \notin \rho (D(G/H)^{H})$. Indeed, let $\omega_{\gra}^{\pm}$ be
the $T$-weight of an equation of $\pi^{-1}(F_{\alpha}^{\pm})$,
then, for any $t\in T^{1}$,
 $(-\omega_{\gra}^{-})(t)=\theta(\omega_{\gra}^{+})(t)=
\omega_{\gra}^{+}(\theta(t))=\omega_{\gra}^{+}(t^{-1})=(-\omega_{\gra}^{+})(t)$,
so
$2\omega_{\gra}^{+}|T^{1}=2\omega_{\gra}^{-}|T^{1}=\omega_{\gra}|T^{1}$,
where $\omega_{\gra}$ is the  fundamental spherical weight  corresponding to $\gra$ (see
\cite{Ru2}, pages 6-8 and \cite{V90} \S3.3-3.4). In particular,
there is $Y_{\gra}\in Div^{G}(X)$ with
$i^{*}(2F_{\alpha}^{+})=i^{*}(2F_{\alpha}^{-})=i^{*}(Y_{\gra})$.
Observe that $Div^{G}(X)_{\mathbb{Q}}$ is a complement   to
$ker(i^{*})$, so  $i^{*}$ is injective over
$Div^{G}(X)_{\mathbb{Q}}$ and
$i^{*}(Pic(X)_{\mathbb{Q}})=i^{*}(Div^{G}(X)_{\mathbb{Q}})$.

Let $(Div\, ^{T}(Z^{c})_{\mathbb{Q}})^{W_{G,\theta}}$ be the subgroup of $W_{G,\theta}$-invariants  in
$Div\, ^{T}(Z^{c})_{\mathbb{Q}}$; we can identify this subgroup with its image in $Pic(Z^{c})_{\mathbb{Q}}$. Moreover, this image is
$i^{*}(Pic(X)_{\mathbb{Q}})=i^{*}(Div\, ^{G}(X)_{\mathbb{Q}})$.

Now, we prove the first statement; suppose  $D$ big. Then we can assume, up to linear
equivalence, that $mD=A+M$ where $m>>0$, $A$ is an ample, $B$-stable divisor
and $M$ is an effective, $B$-stable divisor. Then $mi^{*}(D) =i^{*}(A)
+i^{*}(M) $, with $i^{*}A $ ample. By Theorem \ref{eff cone1} we can write, up to linear equivalence, $M=\sum_{E\in \calN(X)}
a_{E}E+\sum b_{\gra}^{+}F_{\alpha}^{+}+\sum b_{\gra}^{-}F_{\alpha}^{-}$ with positive
coefficients. Thus $i^{*}(M)=\sum_{E\in \calN(X)}
a_{E}i^{*}(E)+\frac{1}{2}\sum (b_{\gra}^{+}+b_{\gra}^{-})i^{*}(Y_{\gra})$ is
effective, so $i^{*}(D)$ is big.

Vice versa, suppose that $i^{*}(D)$ is big and that  $D$ is $G$-stable. Fix  an ample, $G$-stable divisor $A$ on $X$,
then, for $m>>0$, $mi^{*}(D)-i^{*}(A)$ is linearly equivalent to an
effective, $T$-stable divisor $M'$ on $Z^{c}$.
We claim that we can
choose  $i^{*}(mD-A)$ as  $M'$.     Remark that $i^{*}(mD-A)$ is $W_{G,\theta}$-invariant.

Let $D_{1}$ and $D_{2}$ be $T$-stable divisors  on $Z^{c}$ such
that: 1) $D_{1}$ is $W_{G,\theta}$ invariant; 2) $D_{2}$ is linear
equivalent to 0; 3) $D_{1}+D_{2}$ is effective. We claim that $D_{1}$ is
effective. Indeed, suppose by contradiction $D_{1}$ non-effective
and write $D_{1}=\sum_{E\in \calN(X), w\in W_{G,\theta}} a_{E,w}wE|Z$,
$D_{2}=\sum_{E\in \calN(X), w\in W_{G,\theta}} b_{E,w}wE|Z$. Then $D_{2}\neq0$ and
there is a strictly negative $a_{E,w}$. Notice that there is
$(E,w')$ such that $b_{E,w'}\leq0$ because $D_{2}$ is principal.
Thus $a_{E,w'}+b_{E,w'}=a_{E,w}+b_{E,w'}<0$, a contradiction. Thus
there is $M\in Div^{G}(X)_{\mathbb{Q}}$ with $i^{*}M=M'$ and
$mD=A+M$. Moreover, the coefficient of $M$ with respect to any $E$
is equal to the coefficient of $M'$ with respect to $E|Z$, which we
know to be positive.  \end{proof}

\

\begin{rem}\label{rem: controes big toroid} If $D$ is in the kernel of $i^{*}$, then it is not
big. Furthermore, if $\rho$ is not injective,  there are non-big
divisors $D$ with $i^{*}(D)$ big. Suppose by simplicity $X$ toroidal
and let $\alpha_{i}^{\vee}\notin\  \rho(D(G/H)^{H})$, then
$\sum_{j\neq
i}F_{\alpha_{j}}+3F_{\alpha_{i}}^{+}-F_{\alpha_{i}}^{-}$ is not big,
but it is equal to $\sum_{j}F_{\alpha_{j}}+
2(F_{\alpha_{i}}^{+}-F_{\alpha_{i}}^{-})$, where $\sum_{j}F_{j}$ is
big and $(F_{\alpha_{i}}^{+}-F_{\alpha_{i}}^{+})\in ker\, i^{*}$.
\end{rem}

\

Let $D$ be a $B$-stable nef Cartier divisor on a projective symmetric variety. By  Remark \ref{rem: riduzione a G stable}, we can write $D\sim D_{1}+D_{2}$ so that:
i) $D_{1}$ is  $G$-stable, effective and nef; 2) $D_{2}$ is a positive linear combination of the $F_{\alpha}^{+}$, up to exchanging some
$F_{\alpha}^{+}$ with the corresponding $F_{\alpha}^{-}$. Moreover $D$ is big if and only if $D_{1}$ is big. Remark, however, that the previous choice of $F_{\alpha}^{+}$ in $\rho(\alpha^{\vee})$  depends on $D$ and that, if $\rho$ is not injective, we can always find another nef divisor for which such choice does not hold, for example $D_{1}+F_{\alpha}^{-}$.
Thus, to study the bigness of any fixed $B$-stable nef Cartier divisor, we can reduce ourselves to study the bigness of an opportune $G$-stable nef Cartier divisor. Observe that if none  $F_{\alpha}^{+}$  belongs to $\calF(X)$ (for example, if $X$ is toroidal), then $h^{D_{1}+D_{2}}=h^{D_{1}}$.
We say that $D$ satisfies (*) if it is equal to  $D_{1}+D_{2}$, with $ D_{1}$, $D_{2}$ as before. Remark the, given any $B$-stable and nef $D$, we can rename the $F_\gra^+$ so that $D$ satisfies $(*)$.

\begin{thm}\label{bigness e nef} Let $X$ be a projective symmetric variety. Let $D$ be a
nef, $B$-stable Cartier $\mathbb{Q}$-divisor on $X$ which satisfies $(*)$ and let $h$  be
the   piecewise linear function over the support of the
colored fan associated to $D_1$.
Write $h=\{h_{\calC}\}_{\{(\calC,\calF)\}}$ where the $\{(\calC,\calF)\}$
are the maximal colored cones. Then $D$ is big if and only if $(\sum
h_{\calC},R^{\vee})\neq0$ for  each irreducible factor
$R^{\vee}$ of $R_{G,\theta}^{\vee}$. Moreover, if $X$ is toroidal or if $D$  is $G$-stable, then $h$ is also the piecewise linear function  associated to $D$.
\end{thm}

The idea of the proof is the following: first, we reduce to study the pullback of $\calO(D)$ to $Z^c$ by the previous proposition. Then, we study
$vol(W_{G,\theta}\cdot(\sum h_{\calC}))$ to verify if such line bundle is big.

\begin{proof} We can suppose $X$ toroidal and  $\mathbb{Q}$-factorial. Let $s$ be the rank of $G/H$ (equal to the dimension of $Z^c$). In this way the support of $h$ can change, but the weight $\sum
h_{\calC}$ doesn't change. Moreover, we can suppose that $D$ is $G$-stable by Remark \ref{rem: riduzione a G stable}.
By Proposition \ref{bigness of toroidal},  $D$ is big if and only if $i^{*}(D)$ is big, where $i:Z^{c}\rightarrow X$ is the inclusion. Observe
that $i^{*}(D)$ is nef and globally generated. Denoted by $Q(D)$  the polytope
$convex(W_{G,\theta}\{h_{\calC}\})$ (as in \S\ref{sez: picard}),
$i^{*}(D)$ is big if and only if $vol(i^{*}(D)):=i^{*}(D)^{s}=s!\,vol(Q(D))$ is strictly
positive. Notice that the $h_{\calC}$ are antidominant because $D$ is
nef and $G$-stable.

First, suppose that $D$ is big and suppose by contradiction that
there is an irreducible factor $R^{\vee}$ of $R_{G,\theta}^{\vee}$
such that $(\sum h_{\calC},R^{\vee})=0$. Then $(h_{\calC},R^{\vee})=0$ for
each $\calC$ because the $h_{\calC}$ are antidominant. Let $\alpha^{\vee}$
be any simple coroot in $R^{\vee}$, then $(wh_{\calC},\alpha^{\vee})=0$
for all $\calC$ and for all $w\in W_{G,\theta}$.
So $Q(D)$ is contained in the hyperplane $(\,\cdot,\alpha^{\vee})=0$,
thus it has volume 0.

Now, suppose that it is verified the condition $(\sum
h_{\calC},R^{\vee})\neq0$  for each irreducible factor $R^{\vee} $ of $R_{G,\theta}^{\vee}$. Let $s$ be rank of $X$ and let $n$ be the number
of $s$-dimensional colored cones, then $v_{1}=\frac{1}{n}\sum_{(\calC,\calF):\, dim\,\calC=s} h_{\calC}$ belongs
to $Q(D)$, so $convex(W_{G,\theta}v)\subset Q(D)$.
Thus it is sufficient to prove that $vol(convex(W_{G,\theta}v_1))>0$.
Write $\overline{R}_{G,\theta}=I_{1}\sqcup J_{1}$ with
$(v_{1},\alpha)\neq0$ if and only if  $\alpha\in I_{1}$. If $J_{1}$ is not empty, then, by hypothesis,  there are $\alpha\in I_{1}$ and $\beta\in J_{1}$ such
that $(\alpha,\beta)\neq0$. Thus
$(s_{\alpha}v_{1},\beta)=
-(\beta,\alpha^{\vee})(v_{1},\alpha)<0$; moreover $s_{\alpha}v_1\in Q(D)$.
Hence $v_{2}=\frac{2}{3}v_{1}+\frac{1}{3}s_{\alpha}v_{1}$ is
antidominant and belongs to $convex(W_{G,\theta}v_1)$. Write
$\overline{R}_{G,\theta}=I_{2}\sqcup J_{2}$ with
$(v_{2},\alpha)\neq0$ if and only if $\alpha\in I_{2}$.
Notice that
$I_{1}\subsetneq I_{2}$; in particular $J_{2}$ contains no irreducible factor of $R_{G,\theta}$. By induction we can find $m$ such that
$\overline{R}_{G,\theta}=I_{m}$ and $v_{m}$ belongs to
$convex(W_{G,\theta}v_1)$; in particular $v_{m}$ is strictly
antidominant. So $vol(convex(W_{G,\theta}v_{1}))\geq
vol(convex(W_{G,\theta}v_{m}))>0$. \end{proof}

Remark that, if $(G,\theta)$ is indecomposable, we obtain that any non-trivial, $G$-stable, nef,
Cartier $\mathbb{Q}$-divisor is big. This fact can also be proved directly, because in this case the morphism associated to the divisor has to be birational.

%The results of this section can be easily extended to the
%$\mathbb{R}$-divisors.
%

\section{Final remarks}\label{sez: conclusion}

The results of this work cannot be extended to a general spherical
variety. In particular, Theorem \ref{eff cone1} %and \ref{eff cone2} %and \ref{bigg cone}
is false if $H$ has infinite index in $N_{G}(H)$, i.e. $G/H$ is not
sober. This means that the valuation cone $cone(\calN)$ is not strictly
convex. First, the class of two distinct $G$-stable prime divisors
can generate the same halfline. For example
$\mathbb{P}^{1}=\mathbb{C}^{*}\cup\{0\}\cup\{\infty\}$, seen as
toric variety, has two $G$-stable prime divisors and Picard number
1.  Furthermore, it is not clear how to extend the Theorem \ref{eff
cone2} to the horospherical varieties. Consider for example
$\mathbb{P}^{2}=\mP(\mC^2\oplus\mC)$ as completion of $SL_{2}/U$, where $U$ is the group
of upper triangular matrices with diagonal entries equal to 1.  This
variety has one color, one $G$-stable divisor, rank 1 and Picard
number 1. In particular, one can show that, given an equivariant
morphism $SL_{3}/U\rightarrow SL_{3}/H$ onto a spherical space,  $H$
is $U$, $SL_{3}$, a Borel subgroup or the semidirect product of $U$ with a cyclic group (indeed, there is a Borel subgroup such that $U=[B,B]\subset H\subset B$ and $B/U\cong \textbf{k}^*$). We have take these examples from \cite{Br07}, $\S4.1$.

%Furthermore, it is not clear how to extend the Theorem \ref{eff
%cone2} to the horospherical varieties, as shown by the example of the projective cone over $\mP^1$ embedded via the sections of $\calO(n)$.
%%. Consider for example
%%$\mathbb{P}^{2}$ as completion of $SL_{3}/U$, where $U$ is the group
%%of upper triangular matrices with diagonal entries equal to 1.  This
%%variety has one color, one $G$-stable divisor, rank 1 and Picard
%%number 1.
%(We have taken these examples from \cite{Br07}, $\S4.1$).

Also the Theorem \ref{bigness e nef} does not hold for a generic spherical variety.  We can generalizes it in two ways: i) substituting $R_{G,\theta}$ with the spherical root system; ii) substituting $R_{G,\theta}^{\vee}$ with the image of colors. But,  in the first case  the horospherical varieties have  not spherical roots and the statement would be trivially satisfied by all the divisors. In the second case, the flag manifolds have rank zero, so the image of $\rho$ has to be 0. Therefore none divisor can satisfy the statement.  We hope that the    Theorem \ref{bigness e nef} can be generalized to any sober spherical variety. To this aim, it would be useful to prove such theorem using the definition of big divisors based on  $dim H^{0}(X,mL)$. Unfortunately, we  have not succeeded in doing it,  even for the symmetric case.

\section*{Acknowledgments}

The author has been  supported by the M.P.I.M of Bonn.

\end{document}